\documentclass[12pt,a4]{article}
\usepackage{graphicx}
\usepackage{amsmath}
\usepackage{graphicx}
\usepackage{amsfonts}
\usepackage{amssymb}
\usepackage{amsmath,color, amssymb,amsthm, amsfonts, amsgen}
\usepackage{url}
\DeclareGraphicsExtensions{.eps,.bmp,.jpg,.pdf,.mps,.png,.gif}

\usepackage{hyperref}

\usepackage{soul}

\usepackage{color}

\newcommand{\dmn}{\mathop{\rm dom}}

\newcommand{\clo}{\mathop{\rm cl_0}}
\renewcommand{\kappa}{\varkappa}

\newcommand{\Real}{\mathbb R}
\newcommand{\Ntr}{\mathbb N}
\newcommand{\Comp}{\mathbb C}
\newcommand{\eps}{\varepsilon}

\newcommand{\cA}{\mathcal{A}}
\newcommand{\cE}{\mathcal{E}}

\newcommand{\cK}{\mathcal{K}}

\newcommand{\cB}{\mathcal{B}}
\newcommand{\cH}{\mathcal{H}}

\newcommand{\cD}{\mathcal{D}}

\newcommand{\cL}{\mathcal{L}}

\newcommand{\ra}{\rangle}
\newcommand{\la}{\langle}
\newcommand{\ue}{u^\eps}
\newcommand{\lme}{\lambda^\eps}
\newcommand{\mbu}{\hat{u}}
\newcommand{\mbv}{\hat{v}}

\newtheorem{theorem}{Theorem}
\newtheorem{lemma}[theorem]{Lemma}
\newtheorem{proposition}[theorem]{Lemma}
\newtheorem{remark}[theorem]{Remark}%

\numberwithin{equation}{section}

\begin{document}

\title{On eigenvibrations of branched structures with heterogeneous mass density}

\author{Yuriy Golovaty\footnote{Department of Mechanics and Mathematics, Ivan Franko National University of Lviv, Universytetska str., Lviv, 79000, Ukraine,  \url{yuriy.golovaty@lnu.edu.ua}} \and 
Delfina G\'omez\footnote{Departamento Matem\'aticas, Estad\'istica y Computaci\'on,
Universidad de Cantabria, Av.~Los Castros, Santander, 39005, Spain, \url{gomezdel@unican.es}} \and 
Maria-Eugenia P\'erez-Mart\'{\i}nez\footnote{Departamento Matem\'atica Aplicada y Ciencias de la Computaci\'on, Universidad de Cantabria, Av.~Los Castros, Santander, 39005, Spain, \url{meperez@unican.es}}}
\date{}

\maketitle

\noindent
{\bf Abstract:} 
We deal with  a spectral problem for the Laplace-Beltrami operator posed on  a stratified set $\Omega$ which is composed of   smooth surfaces joined along a line $\gamma$, {\em the junction}.
Through this junction we impose  the Kirchhoff-type vertex conditions, which imply the continuity of the solutions  and some balance for normal derivatives, and Neumann conditions on the rest of the boundary of the surfaces.  Assuming that the density is $O(\eps^{-m})$ along small bands of width $O(\eps)$, which collapse  into the  line $\gamma$ as $\eps$ tends to zero, and it is $O(1)$ outside these bands, we address the asymptotic behavior, as $\eps \to 0$, of the eigenvalues and of the corresponding eigenfunctions for a parameter $m\geq 1$.
We also study the asymptotics for high frequencies when $m\in(1,2).$

\smallskip



\section{Introduction}

This section is devoted to the introduction and state of the art of the different mathematical issues arising in the  model under consideration. Let us mention Vibrating systems with concentrated masses (see Section \ref{sec1.1}) and Stratified sets as a generalization of metric graphs (see Section \ref{sec1.2}). Also, in Section \ref{sec1.3}, we describe the main results and  the structure of the paper.

\subsection{Vibrating systems with concentrating masses. A historical review}\label{sec1.1}

Vibrating systems with concentrated masses have been widely studied in the literature of different disciplines such as mechanics, civil engineering and mathematics. As is well known introducing a concentrated mass in a vibrating system may distort the vibrations but also  allow to control them (cf. e.g. \cite[VII.10-VII.14]{SHSP-book}). A concentrated mass is referred to as a  ``small region'' where the density is ``much higher'' than elsewhere. We denote by $\rho^\eps$ the density which is assumed $O(\eps^{-m})$ in this region and $O(1)$ outside, $\eps$ being a small parameter that we shall make to go to zero. The concentrated mass can be centered  at a point (cf. \cite{SHSP-book} and \cite{OleinikBook} for  description of the problem in different frameworks) or at very many points  including homogenization processes (cf. \cite{LoPe-review} and \cite{Chechkin} for different reviews). Also, it can be concentrated  along a manifold; further specifying, along 1-d manifold, cf. \cite{Tchatat, GolovatyLavrenyuk2000, GoGoLoPe-1}  for the first works on the subject,  or a 2-d manifold,  cf. \cite{GoLoPe-masaplano} and references therein. Let us also mention the vectorial models in   \cite{SP-Tchatat,GoLoPe-placa} for instance.

Many different situations may occur  depending on the operators under consideration, the boundary conditions and the value of $m$. A common fact is that depending on  $m$, the high frequencies may play an important role, since they give rise to vibrations of the whole structure, i.e.  {\em global vibrations}, while the low frequencies describe vibrations in reduced surroundings of the concentrated mass, i.e. {\em local vibrations}. But also many different important phenomena appear depending on the range of frequencies in which we move. As regards the low frequencies, let us mention  the {\em asymptotic  infinite   multiplicity}  in \cite{PeM2AS-2005} or the strongly oscillatory behavior of the associated eigenfunctions \cite{NaPe-2018}. Similarly, for the high frequencies,  let us mention the whispering gallery phenomena on interfaces at a microscopic level  or the skin-phenomena, cf \cite{LoPe-review} for precise references.

In all these models, when dealing with the Laplacian operator, a different treatment must be given to the different value of $m$, $m\in (0,2)$ or $m>2$, the case $m=2$ making somehow a threshold for the study, since the localization of the vibrations along points or lines may turn  into a phenomena of interaction between microscopic and macroscopic scales, cf. \cite{GoNaOl-2} and  the review \cite{Golovaty2020M2} for the case of a string with concentrated mass, \cite{OleinikBook} and \cite{GoLoPe-1999} for the case of a concentrated mass in dimensions $3$ and $2$ or \cite{Golovaty2022AS} for the case of  mass concentration  along a curve.

However, in the case where the mass concentration   occurs near a manifold, the value $m=1$ also makes a threshold, cf. \cite{GoGoLoPe-1,GoGoLoPe-2,Golovaty2022AS} for 1-d manifold, and the same applies in the case where the perturbation around a curve comes from  stiffness coefficients \cite{GoNaPe-1,GoNaPe-2}, or potential perturbation \cite{Golovaty2022AA}  (cf. e.g. \cite{Arrieta}  for stationary problems).

Mixing together high mass concentration and  stiffness  is  widely used in reinforcement problems  giving rise to interesting phenomena which   includes an asymptotic concentration of the vibrations (associated to low or high frequencies depending on the boundary conditions) near certain points with particular geometrical characteristics of the curves defining the domain of perturbation, cf.  \cite{GoNaPe-1,GoNaPe-2,GoNaPe-3,GoNaPe-4}.

The spectral problem for the Laplace operator with a perturbed density is also used to describe wave propagation in high-contrast photonic and acoustic media. In this case, the density represents a dielectric constant. In \cite{FK98, KK99}, the spectral properties of a medium in which the dielectric constant is very large near a periodic graph in $\mathbb{R}^2$ were investigated.

Also, it should be mentioned that Steklov type problems with the spectral parameter arising on the boundary condition  appear in a natural way as limits of problems with mass perturbation  (cf. the reviews  \cite{LoPe-review, Girouard} and references therein).

Some of the phenomena above described arise in the problem under consideration, with the additional complication resulting from the geometrical configuration of our problem (cf. Section \ref{sec1.3}) which implies boundary value problems on stratified/ramified sets,  as we describe in Section \ref{sec1.2}.

\subsection{Stratified sets as a generalization of metric graphs. Singularly perturbed problem on graphs}\label{sec1.2}
Boundary value problems for  differential  operators on stratified sets (ramified spaces,  branched structures or open book structures) are widely studied in the literature, cf.
\cite{Lumer1980,Nicaise1988,Nicaise1990,VonBelowNicaise1996,FW04}.
Our problem lies within this framework at least at a local level, and also globally when the surfaces become planes; cf. Figures~\ref{FigSS} and \ref{FigGraph}, \eqref{omegaeps} and reference \cite{VonBelowNicaise1996}.

Boundary value problems on stratified sets are a natural generalization to higher dimensions of similar problems on graphs, see recent preprint \cite{AK2024}, in which the basic concepts of quantum graphs are generalized to the case of stratified sets.
 In recent years, the theory of differential operators on metric graphs has been extensively studied due to its numerous possible applications in physics and solid-state engineering. However, the most interesting application of this theory from the point of view of physics is quantum graphs. Quantum dynamics typically exhibit high complexity, particularly when propagating through branched structures. There is a vast amount of literature on quantum graphs, and readers can refer to \cite{Kuchment2002, BerkolaikoCarlson2006, ExnerKeating2008} and the bibliography therein.
A boundary value problem on a metric graph is a set of differential operators on the edges and some matching conditions for solutions at the graph's vertices. There is a broad set of coupling conditions at the vertices for operators on graphs, in contrast to classical 1D operators. This makes the theory of operators on graphs much richer.
However, a large number of possible vertex conditions leads to the problem of choosing physically motivated ones.
The mathematical approach to building correct mathematical models, in addition to the experimental one, is based on various approximations of processes on graphs. For instance, singular perturbation theory provides an efficient method to find physically motivated point interactions at vertices. Suppose we are interested in the effect of a localized potential or a localized mass density at a vertex. In this case, we must analyze the convergence of the family of singularly perturbed operators. The limit operator will include only vertex interaction conditions that are physically determined (see, e.g.,
\cite{Golovaty2023,ExnerManko2013}).
This article studies a mathematical model that generalizes the vibration of a network of strings with heavy connections.
The articles \cite{GolovatyHrabchak2007, GolovatyHrabchak2010} examine spectral problems related to the Laplace operator on metric graphs. The study focuses on perturbations of the mass density near the vertices.

\subsection{Main results and the structure of paper}\label{sec1.3}

The geometrical configuration of the problem that we broach here is completely different from those treated in the literature. We deal with  a boundary value problem  on $\Omega$, a stratified set which is composed of   smooth surfaces (subsets of Riemanian manifolds) somehow joined along a line $\gamma$, {\em the junction}, near  which the mass perturbation is located (see Figure~\ref{FigSS}).  On this domain   we consider a spectral problem associated with the vibrations of such a stratified set. The operator under consideration is the Laplace-Beltrami operator, the mass perturbation being distributed along small bands   close to the junction which form also a stratified set $\omega^\eps$. These bands  of width $O(\eps)$ collapse  into the  line $\gamma$ as $\eps \to 0$, where we impose  the Kirchhoff-type vertex conditions which imply the continuity of the solutions  and some balance for normal derivatives through $\gamma$.  On the rest of the boundary of the surfaces we impose Neumann conditions.
For an extensive introduction to    boundary value  problems for the Laplace-Beltrami operator for Lipschitz domains in Riemannian manifolds    and  their variational formulations on Sobolev spaces, let us mention \cite{MitreaTaylor}.

As above mentioned, the problem represents a first approach to vibrating models arising in many fields where some reinforcements along junctions   become  essential to control vibrations. Example of such structures  where the models can arise are propellers and turbines (cf. Figure~\ref{FigPropellers}), but also in reinforcements of corners of engineering constructions among  others. To detect which mass gives rise to certain kind of vibrations becomes important in numerous aspects.

Assuming that the density is $O(\eps^{-m})$ in the  stratified set $\omega^\eps$ which may be seen as the edges of a cylinder of radius $O(\eps)$ and length $O(1)$,
see Figure~\ref{FigGraph}, we address the asymptotic behavior, as $\eps$ tends to zero, of  the spectrum of problem  \eqref{PertPrblEq}-\eqref{PertPrblKirh} for a positive parameter $m$.

The model being completely new in the literature, the aim of Section \ref{SecOperatorB} is to determine the spectral properties of eigenvalues and eigenfunctions of the associated self-adjoint bounded operator.  In order to do that,     we  relate its spectrum  with  that of  a Dirichlet-to-Neumann type operator on $L^2(\gamma)$ and of the operator associated to problem \eqref{PrbB00N} that keeps  $\gamma$ fixed, cf. Theorem \ref{LemmaSpectrumB}.
 For fixed $\eps$ and $m$, the  spectrum  of problem \eqref{PertPrblEq}-\eqref{PertPrblKirh} is discrete and we  denote by $\{\lambda_j^\eps\}_{j=1}^\infty$ the set of eigenvalues with the convention of repeated index. In Section \ref{SecAsymp},    by means of matched asymptotic expansions we address the case where $m=1$,   obtaining as limit problem \eqref{LimitPrbEqN}-\eqref{LimitPrbKirch},  a spectral problem in the  stratified set  $\Omega$ with the spectral parameter appearing both  on the partial differential equation and on the junction condition along $\gamma$ relating solutions and normal derivatives  through $\gamma$.    It has also   a discrete spectrum that we denote by $\{\lambda_j\}_{j=1}^\infty$ with a structure described by Theorem \ref{ThLimitSpectrum}. In Section \ref{SecM1}, we show the convergence with conservation of the multiplicity, based on properties from spectral perturbation theory for uniform discrepancies in the operators norm. More specifically, for each $j=1,2,\cdots$, we have
\begin{equation}\label{low1}
  |\lambda^\eps_j-\lambda_j|\le C_j\eps^{1/2},
\end{equation}
where $C_j$ is a constant independent of $\eps$
(see Theorem \ref{TheoremMequals1}).

This implies that the
eigenvalues  $\lambda_i^\eps$ are of  $O(1)$ when $m=1$, and the technique in Section 4 based on asymptotic expansions applies,  with minor modifications,  for $m>1$ and the eigenvalues $\lambda^\eps$ of order $O(1)$, which amounts, in this new case to the {\em high frequencies} and $\lambda^\eps=\lambda_{i(\eps)}^\eps$ where $i(\eps)\to +\infty$ as $\eps \to 0$. Let us explain this in further detail.

Indeed, in Section \ref{SecM>1} we deal with the limit behavior, as $\eps \to 0$, of the eigenvalues $\lambda_i^\eps$ for each fixed $i$. A scaling of these values $\eps^{1-m} \lambda_i^\eps$ along with  the technique in Section \ref{SecM1}, provide us with the limit problem when $m>1$: \eqref{LimitPrbEqNLow}-\eqref{LimitPrbKirchLow} which now has the spectral parameter only on the transmission condition along the junction line $\gamma$. Henceforth there is a mass concentration along $\gamma$, which likely leads to vibrations of this part. We show
\begin{equation}\label{convvpbis}
  |\lambda^\eps_j-\eps^{m-1}\lambda_j|\le C_j\eps^{\alpha(m)},
\end{equation}
where $\alpha(m)=\min\{m-\frac12,2(m-1)\}$.  Obviously, now, $\{\lambda_j\}_{j=1}^\infty$ compose the spectrum of \eqref{LimitPrbEqNLow}-\eqref{LimitPrbKirchLow} (see  Theorem \ref{TheoremLowFrequencies}).

Formula \eqref{convvpbis} determines the order of magnitude of the low frequencies to be $\eps^{m-1}$ and, following the well-known fact that    the high frequencies may accumulate on the whole real positive axis, we look for eigenvalues $\lambda^\eps$ of order $O(\eps^\beta)$ for some $\beta <m-1$ (cf. \cite{LoPe-1997,GoGoLoPe-2,CastroZuazua})  giving rise to other   vibrations that cannot be detected with the low frequencies. This is the aim of  Section \ref{SecM12}, where for the sake of brevity we only address the case of $m\in (1,2)$, leaving the rest of the cases for a forthcoming publication by the authors.

Thus, for $m>1$  the eigenvalues of order $O(1)$ belong to the range of  the high frequencies, and     rewriting the asymptotic expansions in Section \ref{SecAsymp},  with the suitable modifications,  we are lead to the spectrum of operator \eqref{PrbB00}, namely to problem  \begin{equation}\label{PrbB00N}
  -\Delta_\Omega u+Vu=\lambda u  \text{ in } \Omega,\quad
   \partial_n u=0\text{ on } \Gamma,\quad
   u=0 \text{ on } \gamma.
\end{equation}
Henceforth, the corresponding vibrations keep the junction line $\gamma$ fixed.
We show that only the eigenfunctions associated to eigenvalues $\lambda^\eps$ asymptotically near   eigenvalues $\lambda^0$  of problem \eqref{PrbB00N} can be  asymptotically non null in the sense stated by Theorem \ref{ThM12}. We also get results on the total multiplicity of the eigenvalues approaching $\lambda^0$, in the sense stated by Theorem \ref{ThM12bis}. The proof is based on the construction of families of ``almost orthonormal quasimodes'' from the perturbation of eigenvalues.

\section{Statement of problem}

Let us introduce a set that is a bundle of surfaces connected along a curve. Let $\gamma$ be a straight line segment lying on the $x_3$-axis:
\begin{equation*}
    \{x\in \Real^3\colon x_1=0,\;x_2=0,\;0\le x_3\le l\}.
\end{equation*}
Suppose $\Omega_1,\dots,\Omega_K$ is a collection of bounded $C^\infty$-smooth surfaces with the Lipschitz boundaries  embedded
in $\Real^3$ without intersections. We assume that
\begin{equation*}
  \gamma=\bigcap_{k=1}^K \partial\Omega_k,
\end{equation*}
and only the points of $\gamma$ can be common to any pair of these boundaries. Let
\begin{equation*}
 \Omega=\Omega_1\cup\dots\cup\Omega_K.
\end{equation*}
The union $\Omega^*=\gamma\cup\Omega$ can be treated as a stratified set with two strata: the first stratum is the curve $\gamma$ and the second one consists of all surfaces $\Omega_k$.

A function $v$ on $\Omega$ is a collection of functions $\{v_1,\dots,v_K\}$, where $v_k\colon \Omega_k\to \Comp$.
We generally do not assign any values to $v$ on $\gamma$, because the one-sided limits of $v$ at points of $\gamma$ may differ when approached along the different surfaces.
Throughout the paper, $W_2^j(\Omega)$ stands for the Sobolev space of functions belonging to $L_2(\Omega)$ together with their derivatives up to order $j$.  We adhere to the convention that a function $v$ belongs to some space $X(\Omega)$ if $v_k$ belongs to $X(\Omega_k)$ for all $k=1,\dots,K$, i.e.,
\begin{equation}\label{SumSpace}
X(\Omega)=\bigoplus_{k=1}^{K}X(\Omega_k),\qquad \|v\|_{X(\Omega)}=\sum_{k=1}^{K}\|v_k\|_{X(\Omega_k)}.
\end{equation}
Note that the surface $\Omega_k$ inherits a metric by restricting the Euclidean metric to $\Omega_k$. This metric makes $\Omega_k$ into a Riemannian manifold.

Set $\Gamma_k=\partial\Omega_k\setminus\gamma$ and $\Gamma=\bigcup_{k=1}^K \Gamma_k$. We assume that $\Gamma_k$ are $C^2$ curves. Let us introduce two vector fields on $\partial \Omega_k$. The unit outward normal vector to $\Gamma_k$ is denoted by $n_k$, and the unit inward  normal vector to $\gamma$ (as a part of $\partial\Omega_k$) is denoted by $\nu_k$. We combine all  fields $n_k$ into the single normal field $n$ defined on $\Gamma$. In addition, there are $K$ different vector fields $\nu_1,\dots,\nu_K$ on $\gamma$ (see Figure~\ref{FigSS}).

\begin{figure}[ht]
  \centering
  \includegraphics[scale=0.5]{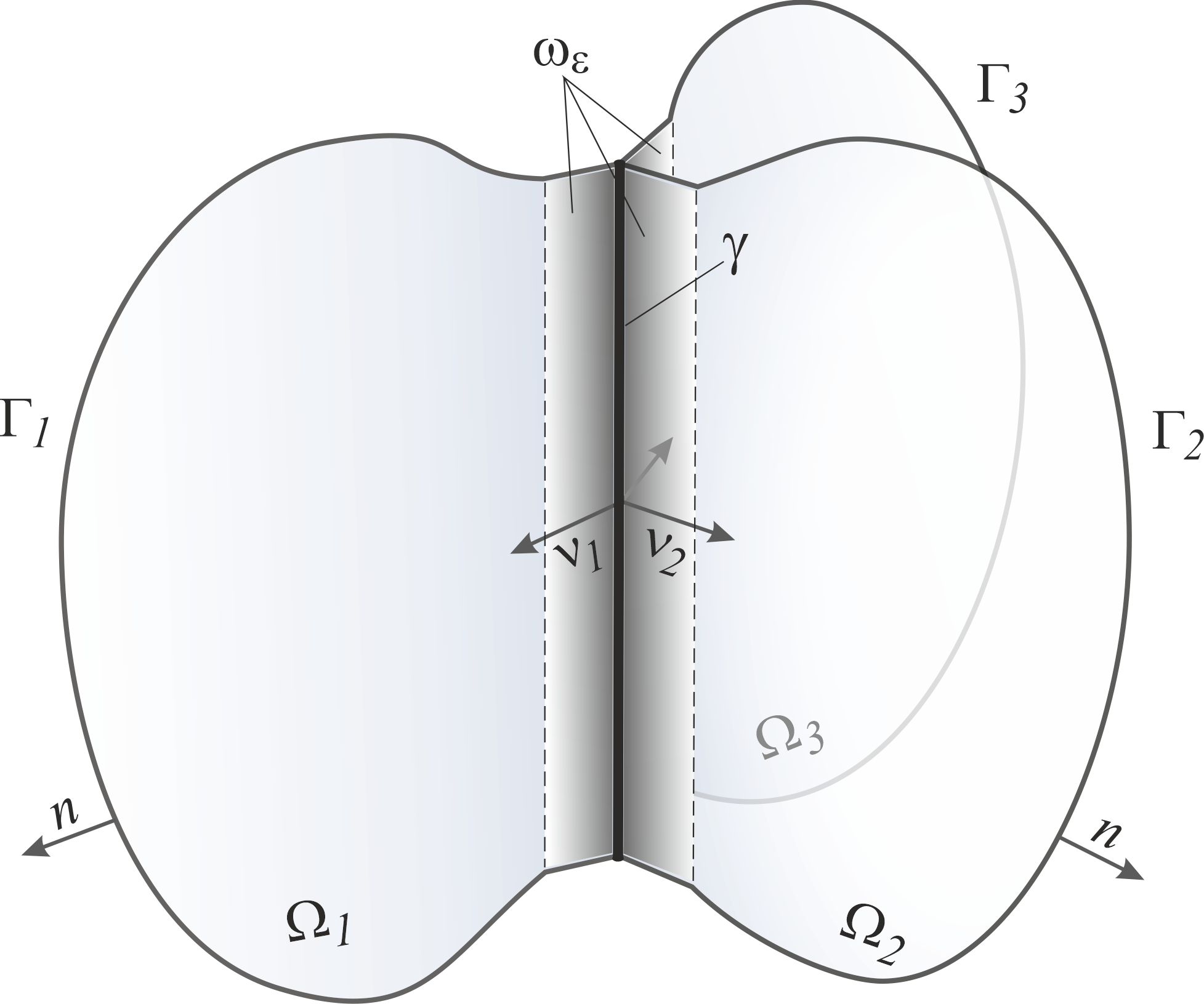}\\
  \caption{Stratified set $\Omega^*$.}\label{FigSS}
\end{figure}

We consider the eigenvalue problem
\begin{align}\label{PertPrblEq}
 -&\Delta_\Omega \ue+(V-\lme\rho^\eps)u^\eps=0 \quad\text{in }\Omega,\\\label{PertPrblNC}
 &\partial_n\ue=0\quad\text{on } \Gamma,\\\label{PertPrblCont}
 &\ue_1=\ue_2=\cdots=\ue_K\quad\text{on } \gamma,\\\label{PertPrblKirh}
 &\partial_{\nu_1}\ue_1+\partial_{\nu_2}\ue_2+\cdots+\partial_{\nu_K}\ue_K=0 \quad\text{on } \gamma.
\end{align}
The operator $\Delta_\Omega$ acts as the Laplace-Beltrami operator $\Delta_{\Omega_k}$ on each $\Omega_k$, i.e.,
\begin{equation*}
  \Delta_\Omega v =\{\Delta_{\Omega_1}v_1,\dots,\Delta_{\Omega_K}v_K\}.
\end{equation*}
The potential $V$ is a real-valued function that belongs to $L^\infty(\Omega)$.  The weight function $\rho^\eps$ describes a highly heterogeneous mass distribution on $\Omega$ as $\eps\to 0$. Let $\omega^\eps$ be the intersection of $\Omega^*$ with the $\eps$-neighbourhood of $\gamma$. We define
\begin{equation*}
  \rho^\eps=
  \begin{cases}
   \phantom{\eps^{-m}} \rho &\text{in } \Omega\setminus \omega^\eps,\\
    \eps^{-m}q^\eps &\text{in } \omega^\eps,
  \end{cases}
\end{equation*}
where $\rho$ and $q^\eps$ are measurable, bounded and positive functions, and $m\ge 1$.
We study the asymptotic behavior as $\eps\to 0$ of the eigenvalues $\lme$ and the eigenfunctions $\ue$ of \eqref{PertPrblEq}-\eqref{PertPrblKirh}.
Equation \eqref{PertPrblEq} is actually the collection of  equations
\begin{equation*}
  -\Delta_{\Omega_k} \ue_k+(V_k-\lme\rho^\eps_k) \ue_k=0\quad\text{in }\Omega_k,\qquad k=1,\dots,K.
\end{equation*}
Conditions \eqref{PertPrblCont}, \eqref{PertPrblKirh} have been inspired by the Kirchhoff vertex conditions that are widely used for the description of string networks and quantum graphs;
also these conditions naturally arise for stratified sets as shown in \cite{CK20}.
These conditions recall transmission conditions.
Condition \eqref{PertPrblCont} ensures continuity of the solution on the whole stratified set $\Omega^*$ while \eqref{PertPrblKirh} can be treated as the tension balance of connected membranes.

Let us introduce some geometric objects and functions above mentioned in further details.

Let $G$ be a compact  star graph  consisting of the vertices $\{a,a_1,\dots,a_K\}$ and the edges $\{e_1=(a,a_1),\dots,e_K=(a,a_K)\}$ meeting at the vertex $a$.  We implement $G$ as a planar metric graph with a metric obtained from the natural embedding of $G$ into $\Real^2_{x_1, x_2}$. Assume that the vertex $a$ coincides with the origin,  other vertices lie on the unit circle $S^1$, and all the edges are radii of $S^1$. Moreover, we assume that the edges $e_1,\dots,e_K$ are drawn in the direction of the normal vectors
$\nu_1,\dots,\nu_K$ respectively. Let $\omega=G\times \gamma$ be the stratified set which consists of $K$ rectangles $\omega_1=e_1\times\gamma,\dots,\omega_K=e_K\times\gamma$ connected along $\gamma$ (see Figure~\ref{FigGraph}).

\begin{figure}[ht]
  \centering
  \includegraphics[scale=0.6]{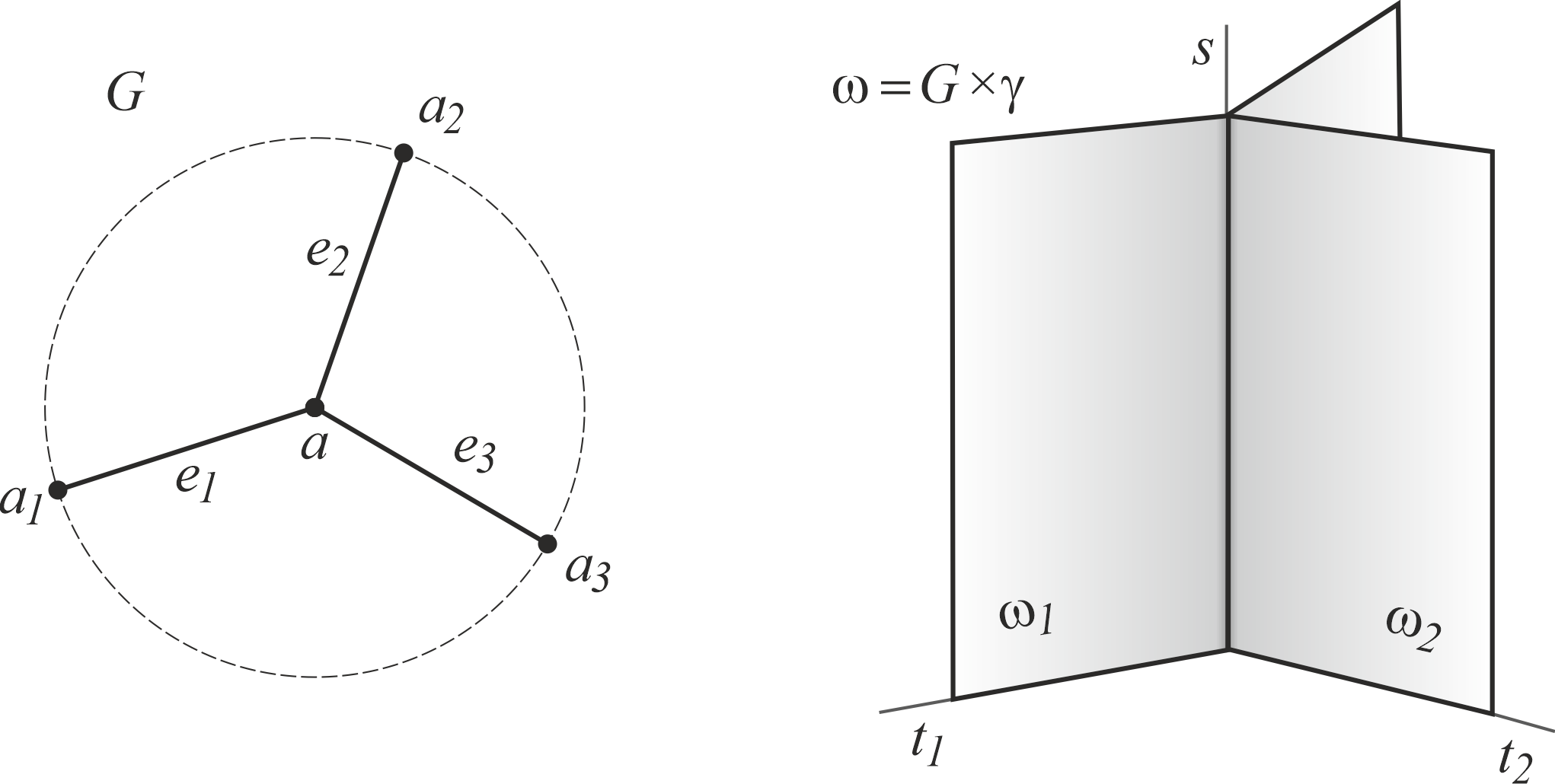}\\
  \caption{The graph $G$ and set $\omega=G\times \gamma$.}\label{FigGraph}
\end{figure}

To keep the mathematics rather simple, we  suppose that the intersection of $\Omega^*$ with the $\eps$-neighbourhood of $\gamma$ has the form
\begin{equation}\label{omegaeps}
\omega^\eps=\{x\in \Real^3\colon (\eps^{-1}x_1, \eps^{-1}x_2, x_3)\in \omega\}.
\end{equation}
This neighborhood is the homothetic to $\omega$ in the $x_1$ and $x_2$ directions of ratio $\eps$. The intersection $\omega^\eps_k=\omega^\eps\cap \Omega_k$ is a rectangle of width $\eps$ and height $l$. We can define the orthogonal  coordinates $(y_k,x_3)$ in $\omega^\eps_k$, where $y_k\in(0,\eps)$ and $x_3\in(0,l)$. Now we can specify the explicit dependence of density $q^\eps$ on a small parameter $\eps$. Let $q\colon \omega\to \Real$ be a measurable, bounded and positive functions. We set
\begin{equation}\label{qeps}
  q_k^\eps(x)=q_k(\eps^{-1} y_k,x_3)\quad\text{in } \omega^\eps_k.
\end{equation}

Similarly, the local coordinate system $(t_k,s)\in(0,1)\times(0,l)$, $t_k=\varepsilon^{-1}y_k$ the stretched  coordinates, appears on each set $\omega_k=e_k\times\gamma$. Here $t_k$ and $s$ are the natural parameters on $e_k$ and $\gamma$ respectively. We say that $\omega$ is equipped with the coordinates $(t,s)$,  meaning that each component $\omega_k$ has its coordinates $(t_k,s)$. We consider $t$ the distance from a point of $\omega$ to $\gamma$.
Also, $f(t,s)$ and $g(y,x_3)$ stand for $(f_1(t_1, s),\dots, f_K(t_K, s))$ and $(g_1(y_1, x_3),\dots, g_K(y_K, x_3))$ respectively.

\section{Spectral properties of the perturbed problem}\label{SecOperatorB}
In this section, we will describe spectral properties of \eqref{PertPrblEq}-\eqref{PertPrblKirh} for a fixed value of $\eps$. We denote by $L_2(h,\Omega)$  the weighted $L_2$-space endowed with the norm
\begin{equation*}
  \|\phi\|_h=(\phi,\phi)_h^{1/2}=\left(\int_\Omega h|\phi|^2\,dS\right)^{1/2},
\end{equation*}
where $h$ is a positive $L^\infty(\Omega)$-function and $dS$ is the volume form on $\Omega$.
We  say that a function $\phi$ is continuous on  $\Omega^*$ if $\phi$ satisfies condition
$\phi_1=\phi_2=\cdots=\phi_K \quad \text{on } \gamma$. In this case,  we write $\phi|_\gamma$ for the common trace of $\phi_k$ on $\gamma$.
We will also write $\cK v$ instead of $\sum_{k=1}^K\partial_{\nu_k}v_k$.
In the space $L_2(h,\Omega)$ we define the operator
\begin{multline*}
  \cB=h^{-1}(-\Delta_\Omega+V)\text{ in } L_2(h,\Omega),\\ \dmn \cB=\big\{\phi\in W_2^2(\Omega)\colon \partial_n \phi=0 \text{ on } \Gamma,  \phi \text{ is continuous on } \Omega^*,\; \cK \phi=0  \big\}.
\end{multline*}
Then eigenvalue problem \eqref{PertPrblEq}-\eqref{PertPrblKirh} is related to the operator $\cA_\eps$, which coincides with the operator $\cB$ for $h=\rho^\eps$, namely
\begin{equation*}
\cA_\eps=\tfrac{1}{\rho^\eps}(-\Delta_\Omega+V)
\end{equation*}
in $L_2(\rho^\eps,\Omega)$, and $\dmn \cA_\eps=\dmn \cB$.

\begin{lemma}\label{LemmaBselfadjoint}
  The operator $\cB$ is closed, self-adjoint, bounded from below, and has a compact resolvent.
\end{lemma}
\begin{proof}
  Given $\phi, \psi\in W_2^2(\Omega)$, we have
\begin{multline*}
    \int_\Omega \Delta_\Omega\phi\, \overline{\psi}\,dS=\sum_{k=1}^K\int_{\Omega_k} \Delta_{\Omega_k}\phi_k \, \overline{\psi_k}\,dS\\=\sum_{k=1}^K\left(\int_{\Gamma_k} \partial_n\phi_k \,\overline{\psi_k}\,d\ell-\int_{\gamma} \partial_{\nu_k}\phi_k \,\overline{\psi_k}\,d\ell\right)-\int_{\Omega} \nabla\phi\cdot
    \nabla\overline{\psi}\,dS,
\end{multline*}
where $d\ell$ is the measure on $\partial \Omega$.
Recall that $\nu_k$ is the inward normal field on $\gamma$.  If we suppose that $\phi$ belongs to the domain of $\cB$, then
\begin{multline*}
  (\cB\phi,\psi)_h-(\phi,\cB^*\psi)_h=-\int_\Omega \Delta_\Omega\phi\, \overline{\psi}\,dS+\int_\Omega\phi\, \Delta_\Omega\overline{\psi}\,dS\\=\sum_{k=1}^K\int_{\Gamma_k} \phi_k \,\partial_n\overline{\psi_k}\,d\ell
  +\sum_{k=1}^K\int_{\gamma} \partial_{\nu_k}\phi_k \,\overline{\psi_k}\,d\ell-\int_{\gamma}\phi \,\cK\overline{\psi}\,d\ell.
\end{multline*}
We see at once that the weakest conditions on $\psi$ under which the equality
\begin{equation*}
(\cB\phi,\psi)_h=(\phi,\cB^*\psi)_h
\end{equation*}
holds for all $\phi\in\dmn\cB$ are $\partial_n \psi=0$  on $\Gamma$,  $\psi_1=\cdots=\psi_K$ and $\cK \psi=0$ on $\gamma$. Therefore $\dmn\cB=\dmn\cB^*$ and $\cB$ is self-adjoint.

Since $V\in L^\infty(\Omega)$, there exists a positive constant $c$ such that $V(x)>-c$ for almost all $x\in \Omega$. Then
\begin{equation*}
  (\cB\phi,\phi)_h=\int_{\Omega}\left(|\nabla\phi|^2+V|\phi|^2\right)\,dS\ge -c\int_{\Omega}|\phi|^2\,dS
 \ge-\frac{c}{h_{min}}\,\|\phi\|^2_h
\end{equation*}
for all  $\phi\in\dmn \cB$, where  $h_{min}=\min_{\Omega} h$. Hence, $\cB$ is bounded from below.

We observe that, for $\lambda\in\rho(\mathcal{B})$, the resolvent $(\cB-\lambda)^{-1}$ is a bounded operator from $L_2(h,\Omega)$ to the domain of $\cB$ equipped with the graph norm. Since the latter space is a subspace of $W_2^2(\Omega)$, it follows that the resolvent is compact as an operator in $L_2(h,\Omega)$. 
\end{proof}

Thus, the spectrum of $\cB$, denoted by $\sigma(\cB)$,  is real discrete, bounded from below, and it  consists of eigenvalues with finite multiplicity. To describe it in more depth, we introduce the sets $\Sigma_\cD$ and $\Sigma_\Theta$ associated with operators $\cD$ and $\Theta(\lambda)$ defined below (cf. Theorem \ref{LemmaSpectrumB}).

Let $M$ be a 2-dimensional, $C^\infty$-smooth, connected, compact, oriented  Rieman\-nian manifold with boundary, and let $\vartheta$ be a non-empty open subset of $\partial M$.  We consider the boundary value problem
\begin{equation}\label{BVPonM}
  -\Delta_M  v+(b-\mu \varrho)v = 0 \text{ in } M,\quad v=\psi\text{ on } \vartheta, \quad \partial_\nu v=0\text{ on }\partial M\setminus \vartheta,
\end{equation}
where $\mu\in\Comp$, $b$ is a real $L^\infty(M)$-function, $\varrho$ is a positive $L^\infty(M)$-function, and $\partial_\nu$ is the inward normal deri\-va\-tive on $\partial M$. Let $\Theta(\mu)$ be the Dirichlet-to-Neumann map
\begin{equation*}
\begin{split}
\Theta(\mu)\psi=\partial_\nu v|_\vartheta, \,\, \dmn \Theta(\mu)=&\big\{\psi\in L_2(\vartheta)\colon v\in W_2^1(M) \text{ and } \partial_\nu v|_\vartheta\in L_2(\vartheta) \\
&\, \mbox{ where $v$ is a solution of \eqref{BVPonM} for given $\psi$}\big\}.
\end{split}
\end{equation*}
This map transforms the Dirichlet data on $\vartheta$ for solutions  into the Neumann ones. It is well-defined for all $\mu$ that do not belong to the spectrum of the operator
\begin{equation*}
    D=\!\varrho^{-1}(-\Delta_M+b),\,\, \dmn D=\!\big\{\phi\in W_2^2(M)\colon \phi=0 \text{ on } \vartheta,\, \partial_\nu\phi=0 \text{ on } \partial M\setminus\vartheta\big\}.
\end{equation*}
For real $\mu$, the operator  $\Theta(\mu)$ is self-adjoint in $L_2(\vartheta)$, bounded from below and has compact resolvent \cite[Th.3.1]{ArendtMazzeo}.
For $k=1,\dots,K$,  we will denote by $\Theta_k(\mu)$ and $D_k$ the Dirichlet-to-Neumann map and the operator $D$ respectively for the case when $M=\Omega_k$, $\vartheta=\gamma$, $b=V_k$, and $\varrho=h_k:=\rho_k^\eps$.

We introduce the operator
\begin{equation*}
\cD=D_1\oplus\cdots\oplus D_K.
\end{equation*}
If $\lambda\not\in\sigma(\cD)$, then the operator
\begin{equation}\label{OperatorThetaH}
  \Theta(\lambda)=\Theta_1(\lambda)+\cdots+\Theta_K(\lambda)
\end{equation}
is well-defined. Moreover, $\Theta(\lambda)$ is self-adjoint in $L_2(\gamma)$, bounded from below and has compact resolvent as it is the sum of operators $\Theta_k(\lambda)$, each of which has these properties. We introduce the set
\begin{equation*}
 \Sigma_\Theta=\big\{\lambda\in \Real\colon  \ker\Theta(\lambda)\neq\{0\}\big\}.
\end{equation*}

Assume that $\lambda$ is an eigenvalue of $\cD$ of multiplicity $r(\lambda)$, and $r_k$ is the multiplicity of $\lambda$ in the spectrum of $D_k$.  Obviously, $r=r_1+\cdots+r_K$. Let  $U_{\lambda, k}$ be the corresponding eigenspace in $L_2(h_k,\Omega_k)$. If $\lambda\not\in \sigma(D_k)$ for some $k$, then $r_k=0$ and the space $U_{\lambda, k}$ is trivial. We consider the subspace
\begin{equation*}
  N_k(\lambda)=\big\{\partial_{\nu_k}u|_\gamma\colon u\in U_{\lambda, k}\big\}
\end{equation*}
of $L_2(\gamma)$ consisting of normal derivatives on $\gamma$ of all the eigenfunctions from $U_{\lambda, k}$. Since linearly independent eigenfunctions give rise to linearly independent normal derivatives on $\gamma$, we have
$\dim N_k(\lambda)=r_k$.
Let us introduce the sum of these spaces
\begin{equation*}
    N(\lambda)=N_1(\lambda)+\cdots+ N_K(\lambda)
\end{equation*}
 and the subset
\begin{equation*}
  \Sigma_{\cD}=\big\{\lambda\in\sigma(\cD) \colon \dim N(\lambda)<r(\lambda) \big\}.
\end{equation*}

\begin{theorem}\label{LemmaSpectrumB}
The spectrum of $\cB$ has the following properties:
\begin{enumerate}
  \item[\rm{(i)}]  $\sigma(\cB)=\Sigma_\Theta\cup\Sigma_{\cD}$.
  \item[\rm{(ii)}]  If $\lambda\in\Sigma_{\cD}$, then $\lambda$ is an eigenvalue of the operator $\cB$ with multiplicity at least $r(\lambda)-\dim N(\lambda)$.
\end{enumerate}
\end{theorem}

\begin{proof}
  \text{(i)} We first prove that $\sigma(\cB)\subset\Sigma_\Theta\cup\Sigma_{\cD}$.   Let $\lambda$ be an eigenvalue of $\cB$ with eigenspace $U_\lambda$. All functions of $U_\lambda$ are solutions of the problem
 \begin{align}\label{ResPrbEq}
 -&\Delta_\Omega u+(V-\lambda h)u=0\;\;\text{in }\Omega,\quad \partial_nu=0\;\;\text{on } \Gamma,\\\label{ResPrbKirch}
 &u_1=u_2=\cdots=u_K,\quad \cK u=0\;\;\text{on } \gamma.
\end{align}
If there exists a non-zero vector $u\in U_\lambda$ such that $u=0$ on $\gamma$, then $u$ is an eigenfunction of $\cD$ with the eigenvalue $\lambda$. Note that the dimension of $N(\lambda)$ cannot exceed $r(\lambda)$, and we have the non-trivial linear combination $\sum_{k=1}^K\partial_{\nu_k}u_k=0$ in this space. Hence $\dim N(\lambda)<r(\lambda)$, and finally $\lambda\in \Sigma_{\cD}$. Otherwise, the trace $\zeta=u|_\gamma$ differs from zero for all non-trivial functions $u\in U_\lambda$. Then
\begin{equation*}
  \cK u=\sum_{k=1}^K\partial_{\nu_k}u_k\big|_\gamma=\sum_{k=1}^K\Theta_k(\lambda)\zeta=\Theta(\lambda)\zeta=0.
\end{equation*}
Hence   $\lambda\in \Sigma_{\Theta}$.

Now we prove the inverse inclusion $\Sigma_\Theta\cup\Sigma_{\cD}\subset\sigma(\cB)$. Assume that $\lambda\in \Sigma_{\Theta}$ and $\zeta$ is a non-zero function belonging to $\ker\Theta(\lambda)$. Let us consider the collection $u=\{u_1,\dots,u_K\}$, where $u_k$ are solutions of the problems
\begin{equation}\label{ProblemUk}
 -\Delta_{\Omega_k} z+(V_k-\lambda h_k) z= 0\;\;\text{in }\Omega_k,\quad \partial_n z=0\;\;\text{on } \Gamma_k, \quad z=\zeta\;\;\text{on } \gamma
\end{equation}
for $k=1,\dots,K$. Then $u$ is an eigenfunction of $\cB$ with the eigenvalue $\lambda$, because $u_1=\cdots=u_K=\zeta$ and $\cK u=\Theta(\lambda)\zeta=0$. Hence, $\lambda\in \sigma(\cB)$.

Next, we suppose that $\lambda\in \Sigma_{\cD}$, i.e., $\lambda$ is an eigenvalue of $\cD$ with multiplicity $r$ such that $\dim N(\lambda)<r$. Let $r_k$ be the multiplicity of $\lambda$ in the spectrum of $D_k$.  If $u_{k1},\dots,u_{k\hskip.6pt r_k}$ are the eigenfunctions of $D_k$ that form a basis in $U_{\lambda, k}$, then the functions $\zeta_{kj}=\partial_{\nu_k}u_{kj}|_{\gamma}$, $j=1,\dots,r_k$, form a basis in $N_k(\lambda)$.
In total, we have $r$ such functions $\zeta_{kj}$ in $N(\lambda)$. If $\dim N(\lambda)<r$, then there exists a non-trivial linear combination
\begin{equation}\label{LinearComb}
  \sum_{k=1}^K\sum_{j=1}^{r_k}\alpha_{kj}\zeta_{kj}=0
\end{equation}
for some constants $\alpha_{kj}$. If we set
\begin{equation*}
  v_k=\sum_{j=1}^{r_k}\alpha_{kj}u_{kj},
\end{equation*}
then $v=\{v_1,\dots,v_K\}$ is an eigenfunction of $\cB$.
Indeed, the functions $v_k$ solve \eqref{ProblemUk} with $\zeta=0$ as a linear combination of eigenfunctions of $D_k$. The continuity condition in \eqref{ResPrbKirch} holds since all $v_k$ vanish on $\gamma$. Next, we have
\begin{equation*}
  \cK v= \sum_{k=1}^K\partial_{\nu_k}v_k=\sum_{k=1}^K\sum_{j=1}^{r_k}\alpha_{kj}\partial_{\nu_k}u_{kj}
  =\sum_{k=1}^K\sum_{j=1}^{r_k}\alpha_{kj}\zeta_{kj}=0,
\end{equation*}
by \eqref{LinearComb}. Hence $\lambda$ is an eigenvalue of $\cB$.

\text{(ii)} If $\dim N(\lambda)=d$, then there exist exactly $r-d$ linearly independent vectors $\alpha=(\alpha_{11},\dots,\alpha_{K,r_K})$ for which \eqref{LinearComb} holds. Therefore we can construct at least $r-d$ linearly independent eigenfunctions of $\cB$. 
\end{proof}

To conclude this section, we recall once again that all the properties of  $\cB$  are also the properties of operators $\cA_\eps$ for a fixed $\eps$.

\section{Formal asymptotics and the limit operator. The case $m=1$}\label{SecAsymp}

In this section, using asymptotic expansions, we will construct a limit operator whose spectrum is the set of limit points for the eigenvalues of \eqref{PertPrblEq}-\eqref{PertPrblKirh} as the small parameter $\eps$ goes to zero.

\subsection{Asymptotics of eigenvalues and eigenfunctions}

We look for the  approximation, as $\eps\to 0$, to an eigenvalue $\lambda^\eps$ and the corresponding eigenfunction $u_\eps$ of \eqref{PertPrblEq}-\eqref{PertPrblKirh} in the form
\begin{align}\label{AsymptoticsL}
&\lambda^\eps=\lambda+o(1),\\\label{AsymptoticsV}
&u^\eps(x)= u(x)+o(1)&&\text{for }x\in \Omega\setminus \omega^\eps,\\\label{AsymptoticsW}
&u^\eps(x)= v(\eps^{-1}y,x_3)+\eps w(\eps^{-1}y,x_3)+o(\eps)&&\text{for }x=(\eps^{-1}y,x_3)\in \omega^\eps.
\end{align}
The function  $u^\eps$ solves \eqref{PertPrblEq} and satisfies \eqref{PertPrblNC} for all $\eps>0$. Since the set $\omega^\eps$ shrinks to $\gamma$ as $\eps\to 0$, the function $u$  must be a solution of the equation
\begin{equation}\label{LimitProblemEq}
 -\Delta_\Omega u+Vu=\lambda \rho u \quad \text{in } \Omega
\end{equation}
that satisfies the boundary condition
\begin{equation}\label{LimitProblemGamma}
\partial_n u=0\quad \text{on } \Gamma.
\end{equation}
Of course, $u$ must also fulfill  appropriate transmission conditions on $\gamma$.
To find these conditions, we will examine equation \eqref{PertPrblEq} in a vicinity of $\gamma$.

The metric in $\omega^\eps_k$ is the Euclidean one, so the Laplace-Beltrami operator $\Delta_{\Omega_k}$ becomes $\partial^2_{y_k}+\partial^2_{x_3}$. In the coordinates $(t,s)$, equation \eqref{PertPrblEq} has the form
\begin{equation*}
  -\eps^{-2}\partial_t^2 u^\eps-\partial_s^2 u^\eps+V(\eps t,s)u^\eps=\lambda^\eps\eps^{-1} q(t,s)u^\eps\quad\text{in }\omega.
\end{equation*}
Here $\partial_t^2$ is the second order derivative along edges of $G$.
Substituting \eqref{AsymptoticsW} into the latter equation and collecting the terms with the same powers of $\eps$ yield
\begin{equation}\label{EqnsVW}
  \partial_t^2v=0,\qquad -\partial_t^2w=\lambda q(t,s)v.
\end{equation}
Obviously, both the functions $v$ and $w$ satisfy Kirchhoff's coupling conditions on $\gamma$:
\begin{align}\nonumber
  &v_1(0,s)=\cdots=v_K(0,s), \quad \sum_{k=1}^K \partial_{t_k}v_k(0,s) =0, \\ \label{KirchW}
  &w_1(0,s)=\cdots=w_K(0,s), \quad
  \sum_{k=1}^K \partial_{t_k}w_k(0,s)=0.
\end{align}

To match the approximations on $\partial\omega^\eps$, we write  $u$ in the local coordinates $(t,s)$:
\begin{align*}
u_k(\eps,s)&=v_k(1,s)+\eps w_k(1,s)+o(\eps),\\ \partial_{\nu_k}u_k(\eps,\tau)&=\eps^{-1}\partial_{t_k}v_k(1,s)+\partial_{t_k}w_k(1,s)+o(1),
\end{align*}
as $\eps\to 0$. Then we have
\begin{gather}\label{FittingUkVk}
  u_k(0,s)=v_k(1,s),\\\label{FittingPdVpdW}
  \partial_{t_k}v_k(1,s)=0,\\\label{FittingPdVk0}
 \partial_{\nu_k}u_k(0,s)=\partial_{t_k}w_k(1,s)
\end{gather}
for all $k=1,\dots,K$. Applying \eqref{PertPrblNC} we also deduce that
\begin{equation}\label{FiitingN}
  \partial_s v(t,0)=0,\quad\partial_s v(t,l)=0,\quad\partial_s w(t,0)=0,\quad\partial_s w(t,l)=0.
\end{equation}

Denote by $\partial G$ the set of vertices $\{a_1,\dots,a_K\}$.
Collecting \eqref{EqnsVW}-\eqref{KirchW}, \eqref{FittingPdVpdW}, and \eqref{FittingPdVk0}, we can now form the problems for $v$ and $w$. The first is the homogeneous boundary value problem in star graph $G$ for the second derivative $\partial^2_t$ depending on parameter $s$:
\begin{align}\label{PrbVEq}
  -&\partial_t^2v=0\;\;\text{in }G\times \gamma,\quad
   \partial_tv=0\;\;\text{on }\partial G\times \gamma,\\\label{PrbVKirch}
  &v_1=\cdots=v_K, \quad \sum_{k=1}^K \partial_{t_k}v_k(0,\cdot) =0 \;\;\text{on }\gamma.
\intertext{The problem for $w$ is the same but already non-homogeneous:}
\label{PrbWEq}
-&\partial_t^2w=\lambda qv\;\;\text{in }G\times \gamma, \quad \partial_tw=\partial_\nu u\;\;\text{on }\partial G\times\gamma, \\\label{PrbWKirch}
 &w_1=\cdots=w_K, \quad \sum_{k=1}^K \partial_{t_k}w_k(0,\cdot) =0 \;\;\text{on }\gamma.
\end{align}
Here, $\partial_tw=\partial_\nu u$ is an abbreviation for the set of conditions \eqref{FittingPdVk0}.

For a fixed $s\in (0,l)$, problem \eqref{PrbVEq}-\eqref{PrbVKirch} has only constant solutions (see \cite{BerkolaikoKuchmentBook}, for details concerning ODE on metric graphs). We put $v(t,s)=\alpha(s)$ and assume that $\alpha\in W_2^{3/2}(\gamma)$, $\alpha'(0)=\alpha'(l)=0$, because of \eqref{FiitingN}.
In view of \eqref{FittingUkVk}, we now obtain
\begin{equation}\label{LimitProblemContinU}
  u_1(0,s)=u_2(0,s)=\cdots=u_K(0,s)=\alpha(s),
\end{equation}
that is to say $u$ must be continuous on $\Omega^*$. So, $v(t,s)=u(0,s)$.

Problem \eqref{PrbWEq}-\eqref{PrbWKirch} is ge\-ne\-rally unsolvable, because the corresponding homogeneous problem has non-trivial solutions.  We will find its solvability conditions, which will simultaneously be another coupling condition on~$u$. Now equation \eqref{PrbWEq} can be written as $-\partial_t^2w=\lambda q(t,s)u(0,s)$. Let us multiply  this equation by an arbitrary function $\phi\in C_0^\infty(\gamma)$ and integrate over $\omega$:
\begin{equation}\label{IntEqW}
  -\int_{\omega} \partial^2_tw(t,s)\overline\phi(s)\,dt\,ds=\lambda\int_{\omega}q(t,s)u(0,s)\overline\phi(s)\,dt\,ds.
\end{equation}
Both sides can be simplified as follows. Integrating by parts yields
\begin{multline*}
  \int_{\omega} \partial^2_tw(t,s)\overline\phi(s)\,dt\,ds=\sum_{k=1}^K\int_0^l\overline\phi(s) ds\int_0^1\partial^2_{t_k}w_k(t,s)\,dt\\
  =\sum_{k=1}^K\int_0^l \left(\partial_{t_k}w_k(1,s)-\partial_{t_k}w_k(0,s)\right)\overline\phi(s)\,ds\\
  =\int_0^l \sum_{k=1}^K\partial_{\nu_k}u_k(0,s)\overline\phi(s)\,ds-\int_0^l\sum_{k=1}^K\partial_{t_k}w_k(0,s)\overline\phi(s)\,ds=\int_\gamma \cK u\,\overline\phi\,d\ell.
\end{multline*}
Above we have used \eqref{PrbWEq} and \eqref{PrbWKirch}.
Next, we write
\begin{equation*}
 \int_{\omega}q(t,s)u(0,s)\overline\phi(s)\,dt\,ds=\int_\gamma \kappa u(0,\cdot)\overline\phi\,d\ell,
\end{equation*}
where the function
\begin{equation}\label{Kappa}
  \kappa(s)=\int_Gq(t,s)\,dt
\end{equation}
describes the total mass of the graph $G_s=G\times\{s\}$. The integral over the graph is the sum of integrals over edges, i.e.,
\begin{equation}\label{KappaAsSum}
\kappa(s)=  \int_Gq(t,s)\,dt=\sum_{k=1}^K\int_{e_k}q_k(t_k,s)\,dt_k.
\end{equation}
Then identity \eqref{IntEqW} becomes
\begin{equation*}
  \int_\gamma (\cK u+\lambda \kappa u)\overline\phi\,d\ell=0\qquad\text{for all } \phi \in C_0^\infty(\gamma).
\end{equation*}
Finally, we get the last condition
\begin{equation}\label{LimitProblemKu}
  \cK u+\lambda \kappa u=0\quad\text{on } \gamma
\end{equation}
on the function $u$, for which we need to formulate the problem.
Combining \eqref{LimitProblemEq}, \eqref{LimitProblemGamma},  \eqref{LimitProblemContinU} and \eqref{LimitProblemKu}, we obtain the limit eigenvalue problem
\begin{align}\label{LimitPrbEqN}
   -&\Delta_\Omega u+Vu=\lambda \rho u \quad \text{in } \Omega,\qquad
   \partial_n u=0\quad \text{on } \Gamma,\\\label{LimitPrbKirch}
   &u_1=u_2=\cdots=u_K,\qquad \cK u+\lambda \kappa u=0\quad\text{on } \gamma,
\end{align}
for the leading terms $\lambda$ and $u$  of asymptotics \eqref{AsymptoticsL} and \eqref{AsymptoticsV}.

\subsection{Properties of the limit operator}

\eqref{LimitPrbEqN}-\eqref{LimitPrbKirch} is a spectral problem where the spectral parameter appears in both, the partial differential equation and the junction condition along $\gamma$. Below,
 we will construct some matrix operator associated with the problem. Let us introduce the space $\cL=L_2(\rho,\Omega)\times L_2(\kappa,\gamma)$ with the inner product
\begin{equation*}
  (\mbu,\mbu)_{\cL}=\int_\Omega\rho|u|^2\,dS+\int_\gamma\kappa|\zeta|^2\,d\ell,
\end{equation*}
for $\mbu=(u,\zeta)^T$ a $2\times 1$ vector function belonging to $\cL$. In this space, we consider the operator
\begin{equation}\label{LimitOp}
  \cA\mbu=
  \begin{pmatrix}
    \rho^{-1}(-\Delta_\Omega u+Vu) \\
    -\kappa^{-1}\cK u
  \end{pmatrix}
\end{equation}
that is defined on the subspace
\begin{equation*}
  \dmn \cA=\big\{(u,u|_\gamma)\colon  u\in W_2^2(\Omega),\;
  u \text{ is continuous on }\Omega^*,\;  \partial_nu=0 \text{ on } \Gamma  \big\}.\\
\end{equation*}
Now problem \eqref{LimitPrbEqN}-\eqref{LimitPrbKirch} can be written in the form
\begin{equation*}
  \cA\mbu=\lambda\mbu, \qquad \mbu\in \dmn\cA.
\end{equation*}

The study of the spectra of the operators $\cA$ and $\cB$ is similar. Therefore, we will only point out some differences without repeating ourselves. Here and subsequently, the operators $D_k$, $\cD$ and $\Theta(\lambda)$ refer to the definitions provided in Section~\ref{SecOperatorB} for the case where $h=\rho$.
Let us introduce the set
\begin{equation*}
 \Lambda_\Theta=\left\{\lambda\in \Real\colon \ker(\Theta(\lambda)+\lambda\kappa I)\neq \{0\}\right\},
\end{equation*}
where $I$ is the identity operator on $L_2(\gamma)$.

\begin{theorem}\label{ThLimitSpectrum}
The spectrum of $\cA$ has the following properties:
\begin{enumerate}
  \item[\rm{(i)}]  It is real discrete, bounded from below, and it  consists of eigenvalues with finite multiplicity.

\item[\rm{(ii)}]   $\sigma(\cA)=\Lambda_\Theta\cup\Sigma_{\cD}$.

 \item[\rm{(iii)}]   If $\lambda\in\Sigma_{\cD}$, then $\lambda$ is an eigenvalue of the operator $\cA$ with multiplicity at least $r(\lambda)-\dim N(\lambda)$.
\end{enumerate}
\end{theorem}

\begin{proof}
First we prove that $\cA$ is self-adjoint, bounded from below, and has a compact resolvent.
Suppose $\mbu\in\dmn\cA$. An easy computation shows that
\begin{equation}\label{AuvuAvFinal}
  (\cA\mbu,\mbv)_\cL-(\mbu,\cA^*\mbv)_\cL
  =\sum_{k=1}^K\int_{\Gamma_k} u_k \,\partial_n\overline{v_k}\,d\ell
  +\sum_{k=1}^K\int_{\gamma} \partial_{\nu_k}u_k (\overline{v_k}-\overline{\eta})\,d\ell,
\end{equation}
for any $\mbv=(v,\eta)^T\in \cL$, provided $v$ belongs to $W_2^2(\Omega)$.
If we suppose that $\partial_nv=0$ on $\Gamma$, the function $v$ is continuous on $\Omega^*$ and $\eta=v|_\gamma$, then the right hand side of \eqref{AuvuAvFinal} vanishes for all $\mbu\in\dmn\cA$.  Furthermore, this is the largest class of vectors $\mbv$  for which this is true.   Hence, $\dmn \cA=\dmn \cA^*$ and $\cA^*$ is self-adjoint.

Next, we have
 \begin{multline*}
 (\cA\mbu,\mbu)_\cL=\int_\Omega(-\Delta_\Omega u+Vu)\overline{u}\,dS-\int_{\gamma}\cK u\,\overline{u}\,d\ell
 =\int_{\Omega}\left(|\nabla u|^2+V|u|^2\right)\,dS\\
 \ge-c\int_{\Omega}|u|^2\,dS
 \ge-\frac{c}{\rho_{min}}\left(\int_{\Omega}\rho|u|^2\,dS+\int_\gamma\kappa|u|^2\,d\ell\right)
 =-\frac{c}{\rho_{min}}\|\mbu\|_\cL^2
 \end{multline*}
for all $\mbu\in\dmn\cA$, where $\rho_{min}=\min_{\Omega} \rho$ and $c$ is a positive constant such that $V(x)\ge -c$ for almost all  $x\in \Omega$. Hence $\cA$ is bounded from below.

The resolvent of $\cA$ is a bounded operator from $\cL$ to $\dmn\cA$. This resolvent is compact as an operator in $\cL$ since $\dmn\cA\subset W_2^2(\Omega)\times W_2^{3/2}(\gamma)\subset \cL$
and the last inclusion is compact.

The rest of the proof runs in the same way as in Theorem~\ref{LemmaSpectrumB}. 
\end{proof}

Two different types of eigenvibrations correspond to the parts $\Lambda_\Theta$ and $\Sigma_{\cD}$ of $\sigma(\cA)$. If $\lambda\in \Sigma_{\cD}$, then the corresponding eigenvector has the form $\mbu_\lambda =(u,0)^T$ and the connection curve $\gamma$ remains unmoved in those vibrations. However, if $\lambda\in \Lambda_\Theta$, then $\mbu_\lambda =(u,\zeta)^T$, where $\zeta$ is a non-trivial solution of the equation
\begin{equation*}
  (\Theta(\lambda)+\lambda\kappa)\zeta=0.
\end{equation*}
This implies that $\gamma$ is involved in the system's vibrations.
Our mathematical model can describe the eigenvibrations of many mechanical systems with complex geometry. For instance, Figure~\ref{FigPropellers} depicts turbine blades and various propellers.
From a physics perspective, the first type would illustrate the oscillation of lighter blades with a fixed shaft, while the second type would be the vibration that also propagates to the shaft.

\begin{figure}[ht]
  \centering
  \includegraphics[scale=0.3]{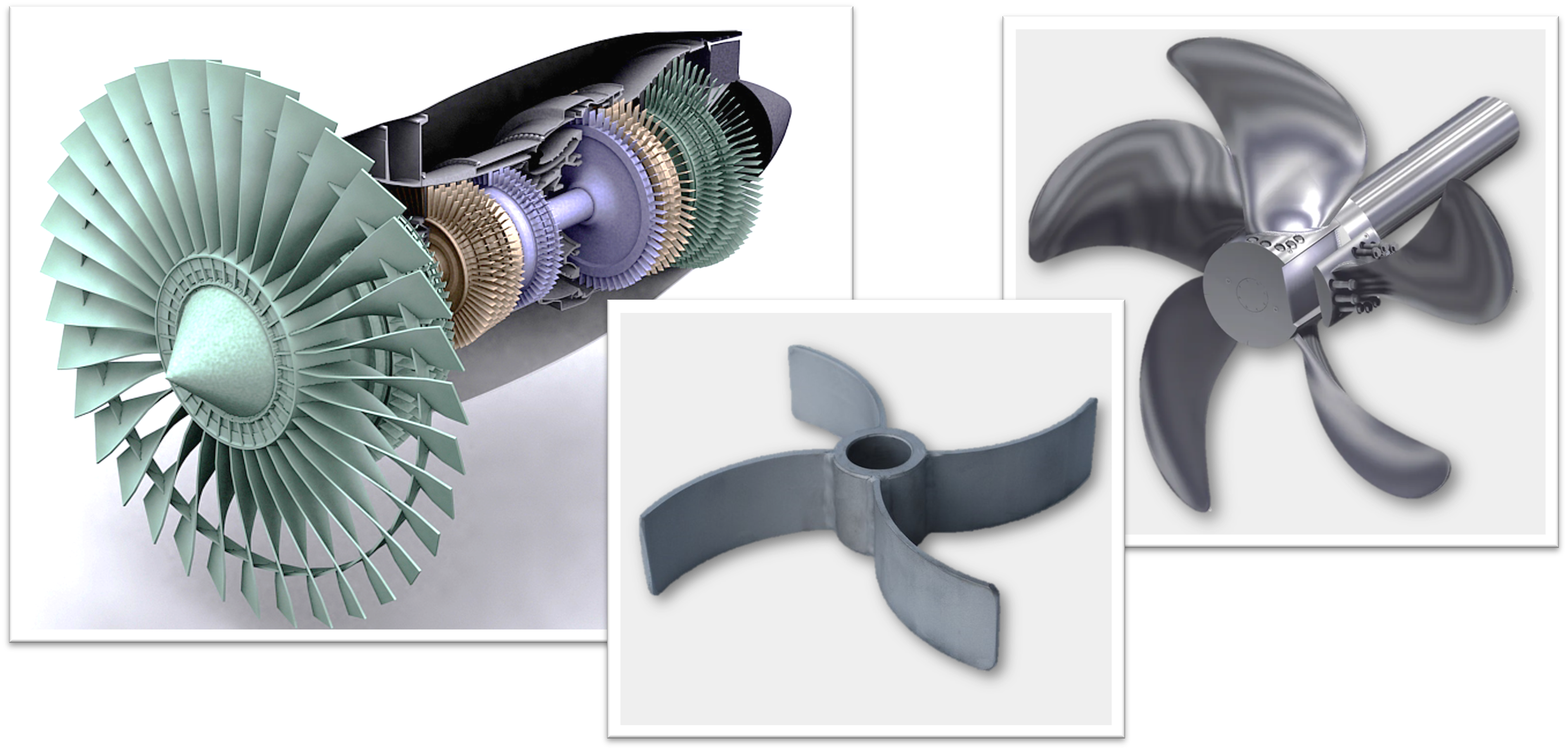}\\
  \caption{Turbines and propellers}\label{FigPropellers}
\end{figure}

\section{Convergence of spectra in the case $m=1$}\label{SecM1}

In this section, we will show that the spectra of the perturbed operators $\cA_\eps$ converge as $\eps\to 0$ to the spectrum of $\cA$. However, each operator $\cA_\eps$ acts in its own Hilbert space $L_2(\rho^\eps,\Omega)$ with the norm depending on the small parameter $\eps$. Therefore, it is convenient to study the convergence of the spectra in terms of the convergence of quadratic forms.

Suppose the potential $V$ is  positive in $\Omega$ and introduce the Hilbert space
\begin{equation*}
  \cH=\{\phi\in W_2^1(\Omega)\colon \phi \mbox{ is continuous in } \Omega^*\}
\end{equation*}
with the inner product $\la\phi,\psi\ra=\int_{\Omega}(\nabla \phi\cdot\nabla \overline{\psi}+V\phi\overline{\psi})\,dS$ and the norm $\|\phi\|=\la\phi,\phi \ra^{1/2}$.
We also define the sesqui\-li\-near forms
\begin{align*}
&a_\eps(\phi,\psi)=\int_{\Omega\setminus\omega^\eps}\rho \phi\overline{\psi}\,dS+\eps^{-m}\int_{\omega^\eps}q^\eps \phi\overline{\psi}\,dS,\\
&a(\phi,\psi)=\int_{\Omega}\rho \phi\overline{\psi}\,dS+\int_\gamma \kappa \phi\overline{\psi}\,d\ell
\end{align*}
acting on the space $\cH$. These forms are associated with compact, self-adjoint operators $A_\eps$ and $A$ in $\cH$ defined as follows
$A_\eps \,  :\, \cH \to \cH, $ $A_\eps \phi=u_\eps$ where $u_\eps$ is the solution of
$$\la u_\eps, \psi \ra =a_\eps(\phi, \psi) \quad \mbox{for all } \psi\in \cH,$$
$A \, :\, \cH \to \cH, $ $A\phi=u$ where $u$ is the solution of
$$\la u, \psi \ra =a(\phi, \psi) \quad \mbox{ for all } \psi\in \cH.$$
In this way, we also have
\begin{equation*}
  \la A_\eps \phi,\psi\ra=a_\eps(\phi,\psi), \qquad \la A \phi,\psi \ra=a(\phi,\psi)\quad \text{for all } \phi,\psi\in\cH
\end{equation*}
(see \cite[III.1]{OleinikBook}, for details).
Then, spectral problems \eqref{PertPrblEq}-\eqref{PertPrblKirh} and \eqref{LimitPrbEqN}-\eqref{LimitPrbKirch} can be written as
\begin{equation*}
  \lambda^\eps A_\eps u^\eps= u^\eps, \qquad \lambda Au=u,
\end{equation*}
respectively.

\begin{theorem}\label{TheoremMequals1}
Let $\{\lambda_j^\eps\}_{j=1}^\infty$ be the increasing sequence of eigenvalues of \eqref{PertPrblEq}-\eqref{PertPrblKirh} for $m=1$, taking multiplicity into account. For problem \eqref{LimitPrbEqN}-\eqref{LimitPrbKirch}, the same sequence of eigenvalues is denoted by  $\{\lambda_j\}_{j=1}^\infty$. Assume the potential $V$ is positive in $\Omega$. Then, for any $n$, we have
\begin{equation*}
  |\lambda^\eps_j-\lambda_j|\le C_j\eps^{1/2}
\end{equation*}
with  some $C_j>0$.
\end{theorem}

Let us first prove some auxiliary estimates.

\begin{proposition}\label{PropositionL2small}
There exists a constant $c>0$ such that
\begin{equation*}
\int_{\omega^\eps}|\phi|^2\,dS\leq c \eps \|\phi\|^2
\end{equation*}
for all $\phi\in\cH$.
\end{proposition}
\begin{proof}
 Let $\phi\in \cH\cap C^1(\Omega)$  where $C^1(\Omega)=\{u|_{\Omega_k}\, : \, u\in C^1(\overline{\Omega_k})\}$ (cf. \eqref{SumSpace}). For any $k=1,2,\ldots K$, we have
\begin{multline*}
|\phi_k(y_k,s)|^2\leq 2 |\phi_k(0,s)|^2+ 2  \Big| \int_0^{y_k}\kern-4pt\partial_{y_k} \phi_k(\tau,s) \, d \tau \Big|^2 \\
\leq 2 |\phi_k(0,s)|^2+ 2  y_k\int_0^{\eps}\kern-2pt |\partial_{y_k} \phi_k|^2 \, d\tau,
\end{multline*}
where $(y_k,s)\in \omega_k^\eps=(0,\eps)\times(0,l)$.
Integrating over $\omega_k^\eps$ and using the Trace Theorem, we get
\begin{multline*}
\int_{\omega^\eps_k}|\phi_k|^2\,dS= \int_{0}^{\eps} \int_0^l |\phi_k(\tau,s)|^2\,ds\,d\tau\\
\leq 2 \eps \|\phi_k\|_{L^2(\gamma)}^2 +\eps^2 \|\nabla \phi_k \|_{L^2(\omega_k^\eps)}^2
\leq c_k\eps \|\phi_k\|_{W^1_2(\Omega_k)}^2.
\end{multline*}
Now let us add all $K$ inequalities.  This completes the proof, since $C^1(\overline{\Omega_k})$ is dense in $W_2^1(\Omega_k)$. 
\end{proof}

\begin{proposition}\label{PropositionDeltaOnGamma}
There exists a positive constant $C$, independent of $\eps$, such that
\begin{equation*}
  \left|\eps^{-1}\int_{\omega^\eps}q^\eps |\phi|^2\,dS-\int_\gamma \kappa |\phi|^2\,d\ell\right|\le C \eps^{1/2}\|\phi\|^2
\end{equation*}
for all $\phi\in\cH$, where $\kappa$ is defined by \eqref{Kappa}.
\end{proposition}
\begin{proof}
As in the previous proposition, it suffices to prove that the estimates
\begin{equation*}
  \left|\eps^{-1}\int_{\omega^\eps_k}q^\eps_k |\psi|^2\,dS-\int_\gamma \kappa_k |\psi|^2\,d\ell\right|\le C_k \eps^{1/2}\|\psi\|^2_{W_2^1(\Omega_k)}
\end{equation*}
hold for all $\psi\in C^1(\overline{\Omega}_k)$ and $k=1,\dots,K$. Here $\psi=\phi_k$ and $\kappa_k(s)=\int_{e_k}q_k(t_k,s)\,dt_k$ (see \eqref{KappaAsSum}).
Let us multiply the obvious equality
\begin{equation}\label{NLforU2}
  |\psi(y_k,s)|^2-|\psi(0,s)|^2=\int_{0}^{y_k} \partial_{y_k} |\psi(\tau, s)|^2\, d\tau
\end{equation}
by the weight function $q_k^\eps(y_k,s)=q_k (\eps^{-1}y_k,s)$ and integrate  along $\gamma$. Then
\begin{multline}\label{EstimateOnGamma}
\left|\int_0^l q_k \big(\eps^{-1}y_k,s) (|\psi(y_k,s)|^2-|\psi(0,s)|^2\big) \, ds\right|
\\
\leq c_1 \int_0^l  \int_{0}^{y_k} \left|\partial_{y_k} |\psi(\tau, s)|^2\right| \, d\tau ds
\leq c_2\int_0^l    \int_{0}^{y_k} |\psi(\tau, s)|\,|\partial_{y_k} \psi(\tau, s)| \, d\tau ds
\\
\leq  c_3 \|\psi\|_{L^2(\omega_k^\eps)}\|\nabla \psi\|_{L^2(\omega_k^\eps)}\leq C \eps^{1/2}\|\psi\|^2,
\end{multline}
in view of Proposition~\ref{PropositionL2small}. We note that
\begin{equation*}
\int_\gamma \kappa_k |\psi|^2\,d\ell=\eps^{-1}\int_0^\eps\int_0^lq_k (\eps^{-1}y_k,s)|\psi(0,s)|^2\,ds\, dy_k.
\end{equation*}
From \eqref{EstimateOnGamma} we get
\begin{multline*}
  \left|\eps^{-1}\int_{\omega^\eps_k}q^\eps_k |\psi|^2\,dS-\int_\gamma \kappa_k |\psi|^2\,d\ell\right|\\
  \leq
  \eps^{-1}\left|\int_0^\eps\int_0^lq_k \big(\eps^{-1}y_k,s) (|\psi(y_k,s)|^2-|\psi(0,s)|^2\big) \,ds\, dy_k\right|\\
 \leq \eps^{-1}\left|\int_0^\eps  C \eps^{1/2}\|\psi\|^2\, dy_k\right|
 \leq  C\eps^{1/2}\|\psi\|^2,
\end{multline*}
which completes the proof. 
\end{proof}

\noindent
{\em Proof of Theorem~\ref{TheoremMequals1}.}
Applying Propositions \ref{PropositionL2small} and \ref{PropositionDeltaOnGamma} yields
\begin{equation*}
    |a_\eps (\phi,\phi)-a(\phi,\phi)|\le \!
    \int_{\omega^\eps}\rho |\phi|^2\,dS+\left|\eps^{-1}\!\!\!\int_{\omega^\eps}q^\eps |\phi|^2\,dS-\!\!\int_\gamma \kappa |\phi|^2\,d\ell\right|\leq c_1\eps^{1/2}\|\phi\|^2
\end{equation*}
for all $\phi\in\cH$.
The latter inequality implies that $A_\eps$ converge to $A$ in the norm and, moreover,
$\|A_\eps-A\|\le c_1 \eps^{1/2}$.
Therefore we conclude that
\begin{equation*}
  \left|\frac1{\lambda^\eps_j}-\frac1{\lambda_j}\right|\le c_j\eps^{1/2},
\end{equation*}
(cf. \cite[III.1]{OleinikBook}) and hence $\lambda^\eps_j\to \lambda_j$ as $\eps\to 0$,  and finally that
\begin{equation*}
  |\lambda^\eps_j-\lambda_j|\le c_j|\lambda_j||\lambda^\eps_j|\eps^{1/2}\le 2c_j|\lambda_j|^2\eps^{1/2}\le C_j\eps^{1/2}
\end{equation*}
for all natural $j$. \qed

\begin{remark}\rm
The operators introduced in Section~\ref{SecOperatorB}, $\cA_\eps$, and the operators generated by forms, $A_\eps$,
share the same set of eigenfunctions. Additionally, the map $\lambda\mapsto \lambda^{-1}$ is a bijection between their spectra. Indeed, all eigenfunctions of $A_\eps$ have higher smoothness and actually belong to the space $W_2^2(\Omega)$. In this case, any eigenfunction $u^\eps$ with an eigenvalue $(\lambda^\eps)^{-1}$ of $A_\eps$
is also an eigenfunction of $\cA_\eps$ with the eigenvalue $\lambda^\eps$ and vice versa, because any weak solution (in the sense of the variational statement) is a strong one.
\end{remark}

\section{Low frequency eigenvibrations in the case $m>1$}\label{SecM>1}

Problem \eqref{PertPrblEq}-\eqref{PertPrblKirh} concerns the eigenvibrations of a propeller with a heavy propeller shaft and relatively light blades. The total mass $M_\eps$ that is concentrated on the shaft has the following asymptotics:
 \begin{equation*}
   M_\eps=\eps^{-m}\int_{\omega^\eps}q^\eps \,dS=\eps^{1-m}\left(\int_\gamma \kappa \,d\ell+o(1)\right), \text{ as } \eps\to 0.
 \end{equation*}
When $m=1$, this mass was finite, but now it goes to infinity as $\eps\to 0$. It is easily seen that
\begin{equation*}
  |a_\eps(\phi,\phi)|\leq c \eps^{1-m}\|\phi\|^2
\end{equation*}
if $m>1$, and
$\|A_\eps\|=O(\eps^{1-m})$ as $\eps\to 0$.
However, the operators $\eps^{m-1}A_\eps$ converge for every $m>1$ and  the limit operator does not depend on $m$. We will define this operator as follows.

Let us consider the eigenvalue problem
\begin{align}\label{LimitPrbEqNLow}
   -&\Delta_\Omega u+Vu=0 \quad \text{in } \Omega,\qquad
   \partial_n u=0\quad \text{on } \Gamma,\\\label{LimitPrbKirchLow}
   &u_1=u_2=\cdots=u_K,\qquad \cK u+\lambda \kappa u=0\quad\text{on } \gamma,
\end{align}
which is similar to \eqref{LimitPrbEqN}-\eqref{LimitPrbKirch}, but in it the weight function $\rho$ is zero.
One interesting aspect of the problem is that the spectral parameter $\lambda$ only appears in the boundary condition.
The operator's eigenfunctions describe the \textit{low frequency eigenvibrations}, which refer to the vibrations of a propeller with weightless blades when all the mass of this vibrating system is concentrated on the propeller shaft. This is best seen in the case of $K=2$, when the stratified set $\Omega^*$ turns into a domain $\Omega\subset\Real^2$ divided by the curve $\gamma$ into two parts, and problem \eqref{LimitPrbEqNLow}-\eqref{LimitPrbKirchLow} can be written as
  \begin{equation*}
    -\Delta u+Vu=\lambda \kappa\delta_\gamma u  \text{ in } \Omega, \quad\partial_n u=0 \text{ on } \partial\Omega,
  \end{equation*}
where the mass density of the vibrating system is Dirac's distribution
\begin{equation*}
  \kappa\delta_\gamma(\psi)=\int_\gamma\kappa\psi\,d\ell, \qquad \text{for all } \psi\in C^\infty_0(\Omega),
\end{equation*}
with the support on $\gamma$.

We will denote by $\Theta$ the operator $\Theta(\lambda)$ from \eqref{OperatorThetaH} in the case when $h=0$.
This operator transforms the Dirichlet data $\zeta$  for the solutions $v_k$ of the problems
\begin{equation}\label{ProblemUkZero}
 -\Delta_{\Omega_k} v_k+V_kv_k = 0\;\;\text{in }\Omega_k,\;\; \partial_n v_k=0\;\;\text{on } \Gamma_k, \;\; v_k=\zeta\;\;\text{on } \gamma,\quad k=1,\dots,K
\end{equation}
to the sum on the normal derivatives $\cK v=\sum_{k=1}^K\partial_{\nu_k}v_k$. This Dirichlet-to-Neumann map is well defined if the corresponding operator $\cD$ is invertible, i.e. all problems \eqref{ProblemUkZero} have only  trivial solutions for $\zeta=0$. Then the  condition $\cK u+\lambda \kappa u=0$ can be written as
\begin{equation*}
  \Theta \zeta+\lambda \kappa \zeta=0.
\end{equation*}

\begin{theorem}\label{Th4}
  If the problem
  \begin{equation}\label{PrbB00}
  -\Delta_\Omega u+Vu=0  \text{ in } \Omega,\quad
   \partial_n u=0\text{ on } \Gamma,\quad
   u=0 \text{ on } \gamma
\end{equation}
has only a trivial solution, then the set of eigenvalues of problem \eqref{LimitPrbEqNLow}-\eqref{LimitPrbKirchLow} coincides with the spectrum of $-\kappa^{-1}\Theta$. This spectrum is real, discrete and consists of eigenvalues with finite multiplicity.
\end{theorem}
\begin{proof}
The operator  $\Theta$ is self-adjoint on $L_2(\gamma)$ and has a compact resolvent \cite[Th.3.1]{ArendtMazzeo}. Therefore, $\kappa^{-1}\Theta$ also possesses these properties, and $\sigma(\kappa^{-1}\Theta)$ is a real discrete set consisting of eigenvalues with finite multiplicity.

If $\lambda$ is an eigenvalue of \eqref{LimitPrbEqNLow}-\eqref{LimitPrbKirchLow} with an eigenfunction $u$, then  $\zeta=u|_\gamma$ differs from zero, because otherwise $u$ would be a solution of \eqref{PrbB00} equal to zero.
Hence, $\lambda$ is an eigenvalue of  $-\kappa^{-1}\Theta$. It is evident that the converse statement is also true. If $\lambda\in \sigma(-\kappa^{-1}\Theta)$ and $\zeta$ is the corresponding eigenfunction, then $\lambda$ is an eigenvalue of \eqref{LimitPrbEqNLow}-\eqref{LimitPrbKirchLow} with the eigenfunction $v=\{v_1,\dots,v_K\}$, where $v_k$ are solutions of \eqref{ProblemUkZero}. 
\end{proof}

\begin{theorem}\label{TheoremLowFrequencies}
Suppose $m>1$ and the potential $V$ is  positive in $\Omega$. Let $\{\lambda_j^\eps\}_{j=1}^\infty$ be the increasing sequence of eigenvalues of \eqref{PertPrblEq}-\eqref{PertPrblKirh}, taking multiplicity into account. For problem \eqref{LimitPrbEqNLow}-\eqref{LimitPrbKirchLow}, the same sequence of eigenvalues is denoted by  $\{\lambda_j\}_{j=1}^\infty$.  Then $\eps^{1-m}\lambda^\eps_j\to\lambda_j$ as $\eps\to 0$, and
\begin{equation}\label{convvp}
  |\lambda^\eps_j-\eps^{m-1}\lambda_j|\le C_j\eps^{\alpha(m)},
\end{equation}
where $\alpha(m)=\min\{m-\frac12,2(m-1)\}$. The constant $C_j$ does not depend on $\eps$.
\end{theorem}
\begin{proof}
As in the proof of Theorem~\ref{TheoremMequals1}, we introduce the sesquilinear form
\begin{equation*}
  a_0(\phi,\psi)=\int_\gamma \kappa \phi\overline{\psi}\,d\ell\qquad \mbox{for all } \phi, \psi \in\cH,
\end{equation*}
and the corresponding self-adjoint operator $A_0 \, :\, \cH \to \cH, $ defined by $A_0\phi=u$ where $u$ is the solution of
$$\la u, \psi \ra =a_0(\phi, \psi) \quad \mbox{for all }\psi\in \cH.$$
In this way, we also have  $\la A_0\phi,\psi \ra=a_0(\phi,\psi)$.

Repeated application of Propositions \ref{PropositionL2small} and \ref{PropositionDeltaOnGamma}  enables us to write
\begin{multline*}
    |\eps^{m-1}a_\eps (\phi,\phi)-a_0(\phi,\phi)|\le
\left|\eps^{-1}\int_{\omega^\eps}q^\eps |\phi|^2\,dS-\int_\gamma \kappa |\phi|^2\,d\ell\right|\\
    +\eps^{m-1}\int_{\Omega\setminus\omega^\eps}\rho |\phi|^2\,dS\leq c(\eps^{1/2}+\eps^{m-1})\|\phi\|^2
\end{multline*}
for all $\phi\in\cH$. So, we see that $\|\eps^{m-1}A_\eps-A\|\le c_1(\eps^{1/2}+\eps^{m-1})$, and therefore
\begin{equation*}
   \left|\frac{\eps^{m-1}}{\lambda^\eps_j}-\frac1{\lambda_j}\right|\le c_j (\eps^{1/2}+\eps^{m-1}).
\end{equation*}
It follows from this estimate that $\eps^{1-m}\lambda^\eps_j$ converge to $\lambda_j$, and
\begin{multline*}
  |\lambda^\eps_j-\eps^{m-1}\lambda_j|\le c_j|\lambda_j||\lambda^\eps_j|(\eps^{1/2}+\eps^{m-1})\\
  \le 2 c_j|\lambda_j|^2\eps^{m-1}(\eps^{1/2}+\eps^{m-1}) \le C_j(\eps^{m-1/2}+\eps^{2(m-1)}),
\end{multline*}
which completes the proof. 
\end{proof}

\section{Asymptotics of upper part of  $\sigma(\cA_\eps)$ in the case $m\in(1,2)$}\label{SecM12}

In the previous section, we described the behavior, as $\eps$ tends to zero, of eigenvalues $\lambda^\eps_j$  for any fixed $j$. The $\{\lambda_j^\eps\}_{j=1}^\infty$ have been ordered in an increasing order  and  the convergence of $\lambda^\eps_j$ to zero is not uniform with respect to the number. Indeed, under the basis of $c_j$ independent of $\eps$ in Theorem \ref{TheoremLowFrequencies}, the constants $C_j$ in inequalities \eqref{convvp} tend to infinity as $j\to \infty$, since $C_j\ge O(|\lambda_j|^2)$. Therefore, even if  $\eps$ is sufficiently small, only a finite number of eigenvalues have the asymptotics given by \eqref{convvp}. For all the other eigenvalues, the value of $C_j\eps^{\beta(m)}$ is the same or larger than $\eps^{m-1}\lambda_j$, and the asymptotic expansion $\lambda^\eps_j=\eps^{m-1}\lambda_j+O(\eps^{\beta(m)})$ is not valid (see Figure~\ref{FigPertubations}\,(a)). This raises the question of the asymptotic behavior of large eigenvalues. We have shown that the spectra of both operators $\cA_\eps$ and $\cA$ can intersect with the spectrum of $\cD$. In this section, we discuss the role of $\cD$ in approximating the upper part of $\sigma(\cA_\eps)$. Under the basis of the normalization of the eigenfunctions in $\cH$,  we show that there are sequences $\{\lambda^\eps\}_{\eps>0}$ of eigenvalues such that $\lambda^\eps\to \mu$ and $\mu>0$, and the corresponding eigenfunctions $u^\eps$ converge towards a  non-zero function $u$ in $\cH$ weakly only if $\mu$ is an eigenvalue of $\cD$ (cf. Theorems \ref{ThM12} and \ref{ThM12bis}). If so, $u$ is the corresponding eigenfunction.

It is assumed that the potential $V$ is positive and $m\in(1,2)$. The quadratic form  $a_\eps(\phi,\phi)$ is continuous with respect to $\eps$, for $\eps\neq 0$.
Therefore, using results on comparison of eigenvalues  and the variational principles (cf. e.g. \cite[I.7]{SHSP-book}), it can be shown  that the eigenvalues $\lambda^\eps_j$  are continuous functions of $\eps\in (0,1]$.
The continuity at zero is a consequence of Theorem \ref{TheoremLowFrequencies}. Let
\begin{equation*}
 \Sigma= \big\{(\eps,\lambda)\colon  \eps\in (0,1), \; \lambda\in \sigma(\cA_\eps)\big\}.
\end{equation*}
This set is the union of all curves in $\Real^2_{\eps,\lambda}$ parameterized by the eigenvalues $\lambda=\lambda^\eps_j$, $\eps\in(0,1)$.
Let $\clo \Sigma$ denote the set of all points $\lambda^*$ such that $(0,\lambda^*)$ belongs to the closure of $\Sigma$, namely, $\lambda^*$ is a limit point of $\lambda^\varepsilon$ as $\varepsilon\to 0.$

\begin{figure}[ht]
  \centering
  \includegraphics[scale=0.45]{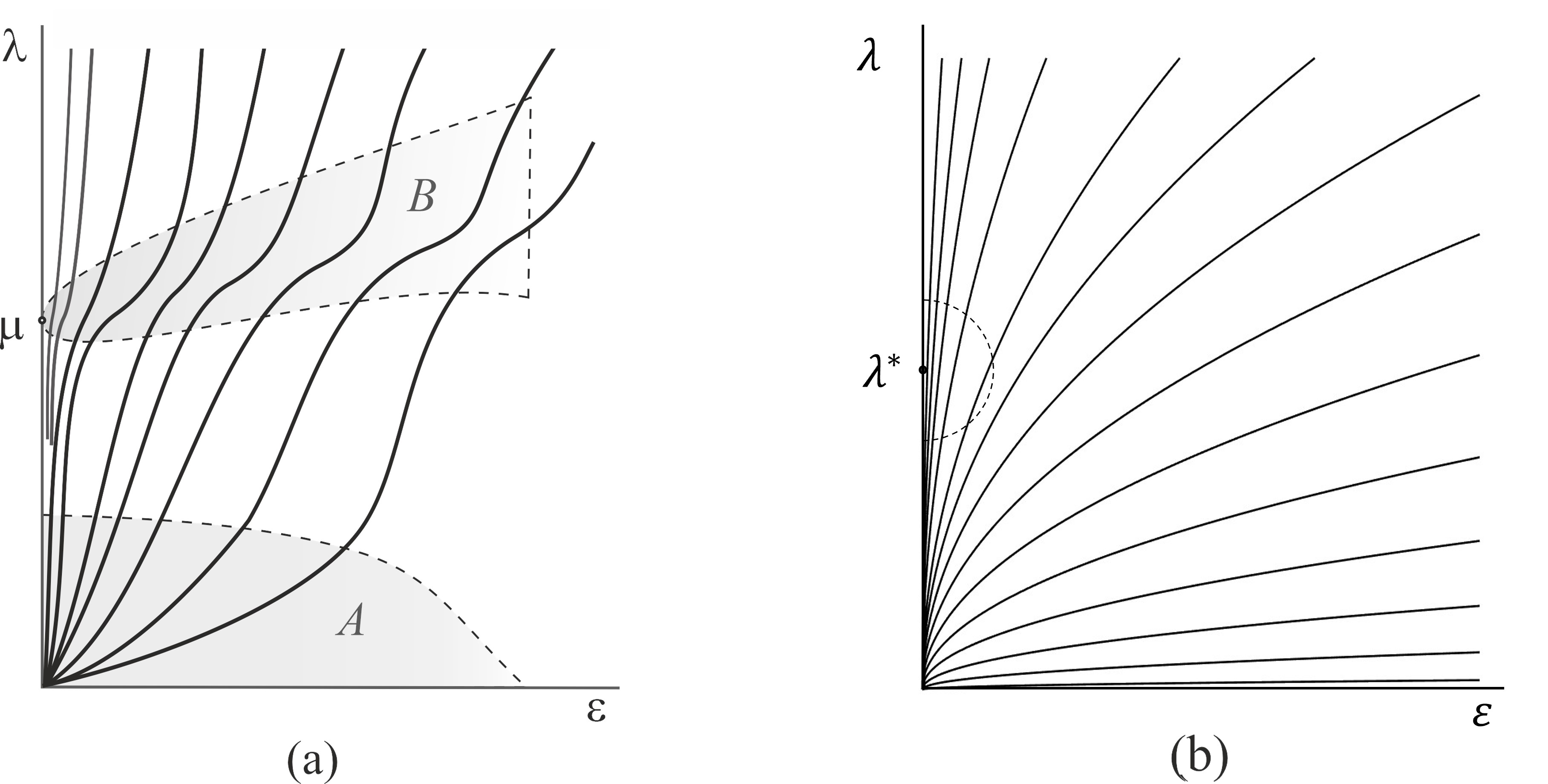}\\
  \caption{(a) The set $\Sigma$: domain $A$ is where  asymptotics \eqref{convvp} holds, while domain $B$ is where the asymptotics is not valid, but other approaches of eigenvalues could exist. (b) An illustration of possible values $\lambda_{j}^{\eps}$ for $m=3/2.$}\label{FigPertubations}
\end{figure}

As a consequence of Theorem 1 in \cite{CastroZuazua}, we claim:

\begin{lemma}\label{LemmaCl0}
  $\clo \Sigma=[0,+\infty)$.
\end{lemma}

 Note that Lemma \ref{LemmaCl0} implies that for each $\lambda^*>0$ there are sequences $\lambda^{\eps_n}_{i(\eps_n)}\to \lambda^*$ as $\eps_n\to 0$ where, on account of Theorem 5, $i(\eps_n)\to +\infty$. It is worth mentioning that the existence of $i(\eps)\to +\infty$ such that the whole sequence $\lambda^{\eps}_{i(\eps )}\to \lambda^*$ could be obtained by means of the corresponding spectral families (cf. \cite{LoPe-1997} for the technique  and \cite{GoGoLoPe-2} for further explanations and references). For the sake of completeness, Remark \ref{RemCl0} contains a formal proof based on a graphic  for the specific order $\lambda_j^\eps=O(\eps^{m-1})$ with $m\in (1,2)$.

\begin{remark}\label{RemCl0}\rm
All eigenvalues $\lambda^\eps_j$ are positive, therefore $\clo \Sigma\subset [0,+\infty)$.
Additionally, $\lambda=0$ belongs to $\clo \Sigma$ according to Theorem~\ref{TheoremLowFrequencies}.
As above mentioned, given $j\in \Ntr$,  the eigenvalue $\lambda^\eps_j$ is a continuous function of $\eps$ that goes to zero as $\eps\to 0$, $\lambda^\eps_j\approx \eps^{m-1} \lambda_j$.  Also $\lambda_j\to +\infty$ and,  for fixed $\eps=\eps_0$,   $\lambda^{\eps_0}_j\to +\infty$  as $j\to + \infty$.
If, contrary to our claim, some point $\lambda^*>0$ is not included in  $\clo \Sigma$,  then,  for all sequences $\eps_n\to 0$,  none subsequence, still denoted by $\eps_n,$  of eigenvalues $\lambda^{\eps_n}_{i(\eps_n)}$can converge towards $\lambda^*$ and  there will likely exists  a neighborhood of $(0,\lambda^*)$, $B_{\lambda^*},$  that is free of points of $\Sigma$.
Based on what is shown in Figure~\ref{FigPertubations}\,(b) we conclude that for sufficiently large $j$ we can find $\eps_j$ sufficiently small such that $\lambda_{j}^{\eps_j}\in B_{\lambda^*}$. This contradicts the assumption.
\end{remark}

The lemma states that any positive number can be approximated by a sequence of eigenvalues of $\cA_\eps$. However, there is a difference  between the spectrum of $\cD$ and the other points in $\clo \Sigma$. This distinction can only be explained by the behavior of the corresponding eigenfunctions.

Let $\cE$ be a subset of the interval $(0,1)$ for which zero is a limit point. We also introduce the space
\begin{equation*}
   \cH_0=\{\phi\in\cH\colon \phi=0 \text{ on }\gamma\}.
\end{equation*}

\begin{theorem}\label{ThM12}
  Assume $m\in(1,2)$. Let $\{\lambda^\eps\}_{\eps\in\cE}$ be a sequence of eigenvalues of $\cA_\eps$ and $\{u^\eps\}_{\eps\in\cE}$ be a sequence of the corresponding eigenfunctions, normalized by $\|u^{\eps}\|=1$. Suppose that $\lambda^\eps$ converge to some positive value $\lambda$ as $\cE\ni\eps\to 0$ and $u^\eps\to u$ in $\cH$ weakly.
 \begin{enumerate}
     \item[\rm{(i)}] If $\lambda\notin \sigma(\cD)$, then $u=0.$
     \item[\rm{(ii)}]   If the limit function $u$ is not equal to zero, then $\lambda$ is an eigenvalue of $\cD$  and $u$ is an  eigenfunction associated with $\lambda$.
  \end{enumerate}
\end{theorem}
\begin{proof}
First, we prove that $u^{\eps}|_{\gamma}\to 0$ in $L_2(\gamma)$, as $\cE\ni\eps\to 0$.
  An eigenfunction $u^\eps$ of problem  \eqref{PertPrblEq}-\eqref{PertPrblKirh} satisfies the identity
 \begin{equation}\label{UepsIdentity}
   \int_{\Omega}\left(\nabla u^\eps\cdot\nabla\overline\phi +Vu^\eps \overline\phi \right)\,dS=\lambda^\eps\int_{\Omega}\rho^\eps u^\eps \overline\phi\,dS \quad \text{for all }\phi\in \cH.
 \end{equation}
When considering the normalized eigenfunction, this identity gives
 \begin{equation*}
   \int_{\Omega\setminus\omega^{\eps}}\rho |u^{\eps}|^2\,dS+{\eps}^{-m}\int_{\omega^{\eps}}q^{\eps} |u^{\eps}|^2\,dS \, =\frac{1}{\lambda^\eps}.
 \end{equation*}
From this, we immediately estimate that
 \begin{equation*}
\eps^{-1}\int_{\omega^{\eps}} |u^{\eps}|^2\,dS\le c_1 \eps^{m-1}.
 \end{equation*}
Next, by applying \eqref{NLforU2} for $u^{\eps}$ instead of $\psi$ and repeating the same computation as in the proof of Proposition~\ref{PropositionDeltaOnGamma}, we obtain
\begin{equation}\label{norm}
  \|u^\eps\|_{L_2(\gamma)}^2=\int_\gamma|u^{\eps}|^2\,d\ell=\eps^{-1}\int_{\omega^\eps}|u^{\eps}(0,\cdot)|^2\,dS\le c_2(\eps^{m-1}+\eps^{1/2}).
\end{equation}
Hence, $u^{\eps}|_{\gamma}$ converge to zero in $L_2(\gamma)$, and moreover
\begin{equation}\label{conver}
  \eps^{-m}\int_{\omega^{\eps}}q^{\eps} u^\eps \overline\phi\,dS\to 0,\qquad\cE\ni\eps\to 0,
\end{equation}
provided $\phi\in \cH_0$ and $m\in(1,2)$. By passing to the weak limit in \eqref{UepsIdentity}, we obtain that
\begin{equation}\label{UIdentity}
   \int_{\Omega}\left(\nabla u\cdot\nabla\overline\phi +Vu \overline\phi \right)\,dS=\lambda \int_{\Omega}\rho u \overline\phi\,dS \qquad\text{for all }\phi\in \cH_0.
\end{equation}
Therefore, $u$ is either an eigenfunction of $\cD$ with the eigenvalue $\lambda$ or zero, since $u|_{\gamma}=0$. 
\end{proof}

The theorem we have just proved does not guarantee the existence of convergent sequences $\lambda^\eps\to \lambda$ and $u^\eps\to u$ such that $u$ is not zero if $\lambda$ belongs to $\sigma(\cD)$. This fact will be demonstrated constructing the so-called quasimodes.
We refer to \cite{VishikLyusternik} for the  proof of Lemma~\ref{Lazutkinlema}.

\begin{lemma}\label{Lazutkinlema}
Let $L : H \longrightarrow H$ be a linear, self-adjoint, positive and compact
operator on a separable Hilbert space $H$ with domain $D(H)$. Let
$v \in D(H)$, with $\|v\|_H=1$ and $\mu, \; r>0$ such that
$\|Lv-\mu v\|_H \leq r$. Then, there exists an eigenvalue
$\mu^*$ of $L$ satisfying $|\mu - \mu^*|\leq r$.
Moreover, for any $d>r$,
 there is $v^*\in H$, with $\|v^*\|_H=1$, $v^*$ belonging to the eigenspace
associated with the eigenvalues of the operator $L$ lying on
the interval $[\mu -d,\mu +d]$ and such that
$$\|v-v^*\|_H \leq 2rd^{-1}.$$
\end{lemma}

The couple $(\mu, v) \in \mathbb{R} \times H$ such that $\|L v-\mu v\|_H \leq r$ and $\|v\|_H=1$ is called a quasimode of the operator $L$ with error $r$.
If $r=0$, then $\mu$ is an eigenvalue of $L$ with the normalized eigenfunction $v$. Otherwise, as stated Lemma \ref{Lazutkinlema}, given a quasimode with error $r,$ the interval $[\mu-r, \mu+r]$ contains at least one eigenvalue $\mu^*$ of $L$.

It should be noted that no assertion can be made about the relative closeness of the quasimode $v$ to a true eigenvector $v^*$. The only fact that can be stated is that
\begin{equation}\label{EVminusV}
  \|E(\Delta)v-v\|_H\le rd^{-1},
\end{equation}
where $\Delta=[\mu-d, \mu+d]$ and $E(\Delta)$ is the spectral projection of $L$ corresponding to $\Delta$. If $\Delta$ contains only one simple eigenvalue $\mu^*$ of $L$, then  there exists a normalized eigenvector $v^*$ such that
\begin{equation}\label{VminusV*}
  \|v-v^*\|_H\le 2rd^{-1},
\end{equation}
since $E(\Delta)=(v,v^*)_H \:v^*$  (see \cite{VishikLyusternik} and \cite{Lazutkin}, for details).

A family of quasimodes $\{(\mu, v_1), \ldots,(\mu, v_J)\}$ with error $r$  is said to have a deviation from orthogonality $\theta$ if $\left|\left(v_i, v_j\right)_H-\delta_{ij}\right| \leq \theta$
for all $i,j=1, \ldots, J$, where $\delta_{ij}$ is the Kronecker delta.
We refer to \cite{Lazutkin} for the  proof of Lemma~\ref{PropQuasimodes}.

\begin{lemma}\label{PropQuasimodes}
Let $\{(\mu, v_1), \ldots,(\mu, v_J)\}$ be a family of quasimodes of the operator $L$ with error $r$ and deviation from orthogonality $\theta$. If $rd^{-1}+\theta<J^{-1}$, then  $L$ has eigenvalues on the interval $[\mu-d, \mu+d]$ with a total multiplicity of $J$.
\end{lemma}

Let us construct quasimodes for the operator $A_\eps\colon\cH\to\cH$ introduced in Section~\ref{SecM1}. We consider the pair $(\lambda^{-1}, u)$, where $\lambda$ is an eigenvalue of $\cD$ and $u$ is the corresponding normalized eigenfunction. We need to evaluate whether the norm $\|A_\eps u-\lambda^{-1} u\|$ is small as $\eps$ tends to zero.

It is observed that $u$ satisfies identity \eqref{UIdentity} for functions $\phi\in\cH_0$, but for the test functions from $\cH$ the following identity holds:
\begin{equation}\label{UIdentityOnH}
   \int_{\Omega}\left(\nabla u\cdot\nabla\overline\phi +Vu \overline\phi \right)\,dS+\int_{\gamma} \cK u \,\overline\phi\,d\ell=\lambda \int_{\Omega}\rho u \overline\phi\,dS \qquad\text{for all }\phi\in \cH.
\end{equation}
In addition, the eigenfunction $u$, as we noted above, belongs to $W_2^2(\Omega)$. Due to the Sobolev embedding  $W_2^2(\Omega_k)\hookrightarrow C^{0,\eta}(\Omega_k)$, valid for $\eta\in(0,1)$
(cf. \cite[1.27, 6.2]{Adams}), we have that
$$|u(x)|\le c |x|^\eta$$ in a vicinity of $\gamma$, because of $u|_\gamma=0$. Combining this with Proposition~\ref{PropositionL2small}, we have
\begin{equation}\label{EstUphi}
  \left|\int_{\omega^\eps}q^\eps u\overline\phi\,dS\right|\le c_1\max_{x\in\omega_\eps}|u(x)|\int_{\omega^\eps} |\phi|\,dS\le
  c_2\eps^{\eta+1}\|\phi\|.
\end{equation}
Applying \eqref{UIdentityOnH} and \eqref{EstUphi}, we deduce
\begin{multline*}
   \la A_\eps u-\lambda^{-1}u,\phi\ra
   =\int_{\Omega\setminus\omega^\eps}\rho u\overline{\phi}\,dS+\eps^{-m}\int_{\omega^\eps}q^\eps u\overline{\phi}\,dS
   \\
   -\lambda^{-1}\int_{\Omega}(\nabla u\cdot\nabla \overline{\phi}+Vu\overline{\phi})\,dS
   =\lambda^{-1}\int_{\gamma} \cK u \,\overline{\phi}\,d\ell +\eps^{-m}\int_{\omega^\eps}q^\eps u\overline{\phi}\,dS-\int_{\omega^\eps}\rho u\overline{\phi}\,dS
   \\
   =\lambda^{-1}\int_{\gamma} \cK u \,\overline{\phi}\,d\ell +O(\eps^{\eta+1-m}) \quad \text{as } \eps\to 0.
 \end{multline*}
Hence, the pair $(\lambda^{-1}, u)$ is not the best candidate for a quasimode, because the vector $A_\eps u-\lambda^{-1}u$ has a large norm in $\cH$. However, we will improve it now.

Let us assume that  $\kappa$ (namely, $q$) is sufficiently smooth and  there is  $g$ from $\cH$ such that \begin{equation}\label{gg}  g(y,s)=(\lambda\kappa(s))^{-1}\cK u(0,s) , \mbox{ in a neighborhood of $\gamma$ and $g\in \cH$.}\end{equation}
Then
\begin{equation}\label{LambdaKappa}
  \left|\eps^{-1}\int_{\omega^\eps}q^\eps g\overline{\phi}\,dS-\lambda^{-1}\int_{\gamma} \cK u \,\overline{\phi}\,d\ell\right|\le c\eps^{1/2}\|\phi\|,
\end{equation}
by Proposition~\ref{PropositionDeltaOnGamma}. We set $w_\eps=u-\eps^{m-1}g$ and consider a new pair $(\lambda^{-1}, w_\eps)$.
Repeating the previous argument and using \eqref{LambdaKappa} leads to the estimate
\begin{multline*}
   \left|\la A_\eps w_\eps-\lambda^{-1}w_\eps,\phi\ra\right|\le \left|\eps^{-1}\int_{\omega^\eps}q^\eps g\overline{\phi}\,dS-\lambda^{-1}\int_{\gamma} \cK u \,\overline{\phi}\,d\ell\right|
    \\
    +\eps^{-m}\left|\int_{\omega^\eps}q^\eps u\overline{\phi}\,dS\right|
    +\left|\int_{\omega^\eps}\rho u\overline{\phi}\,dS\right|+\eps^{m-1}\left|\int_{\Omega\setminus\omega^\eps}\rho g\overline{\phi}\,dS\right|
    \\
    +\lambda^{-1}\eps^{m-1}|\la g,\phi\ra|\le C(\eps^{m-1}+\eps^{1/2}+\eps^{\eta+1-m})\|\phi\|.
 \end{multline*}
Finally, we have
\begin{equation*}
 \left|\la A_\eps w_\eps-\lambda^{-1}w_\eps,\phi\ra\right|\le C \eps^{ \beta(m,\eta)}\|\phi\|,
\end{equation*}
where
\begin{equation}\label{BetaNM}
  \beta(m,\eta)=
  \begin{cases}
    m-1 & \text{if } m\in(1,1+\frac{\eta}2],\\
    \eta-m+1 & \text{if } m\in(1+\frac{\eta}2,\eta+1).
  \end{cases}
\end{equation}
We can see that for any $m$ as close to $2$ as possible, there exists $\eta\in(0,1)$ such that $ \beta(m,\eta)$ is positive. Hence, $\|A_\eps w_\eps-\lambda^{-1}w_\eps\|\le C \eps^{ \beta(m,\eta)}$, and therefore $(\lambda^{-1}, w_\eps)$ is a quasimode of $A_\eps$ with error of order $O(\eps^{ \beta(m,\eta)})$, as $\eps\to 0$.

Let $\lambda$ be an eigenvalue of $\cD$ with multiplicity $J$. In the corresponding eigen\-space~$U_\lambda$, we can choose a basis
$\{u^{(1)},\dots,u^{(J)},\}$ such that
\begin{equation*}
  \lambda\int_\Omega\rho u^{(i)}\overline{u^{(j)}}\,dS=\delta_{ij}\qquad \text{for } i,j=1,\dots,J.
\end{equation*}
Then this basis is orthonormal in the space $\cH$, i.e., $\la u^{(i)}, u^{(j)}\ra=\delta_{ij}$. We can construct the family of quasimodes
\begin{equation*}
  w^{(1)}_{\eps}=u^{(1)}-\eps^{m-1}g^{(1)},\dots,w^{(J)}_{\eps}=u^{(J)}-\eps^{m-1}g^{(J)}
\end{equation*}
with error of order $O(\eps^{ \beta(m,\eta)})$. In addition, this family has a deviation from orthogonality of the order $O(\eps^{m-1})$. Indeed, for every $i, j=1,\dots,J$ we have
\begin{multline*}
  \la w^{(i)}_{\eps}, w^{(j)}_{\eps}\ra-\delta_{ij}=\la u^{(i)}-\eps^{m-1}g^{(i)}, u^{(j)}-\eps^{m-1}g^{(j)}\ra
  -\la u^{(i)}, u^{(j)}\ra
  \\
  =-\eps^{m-1}\big(\la u^{(i)},g^{(j)}\ra+\la g^{(i)},u^{(j)}\ra\big)+\eps^{2(m-1)}\la g^{(i)},g^{(j)}\ra=O(\eps^{m-1}),
  \quad\text{as }\eps\to 0.
\end{multline*}

Lemma~\ref{PropQuasimodes} and  estimates \eqref{EVminusV}, \eqref{VminusV*} will now be applied  to construct a family of quasimodes by setting   $d=2J C\eps^{ \beta(m,\eta)}$, $\theta=c\eps^{m-1}$ and $r=C\eps^{ \beta(m,\eta)}$.
The condition $rd^{-1}+\theta<J^{-1}$ is met because the following inequality
\begin{equation*}
 (2J)^{-1}+c\eps^{m-1}<J^{-1}
\end{equation*}
holds for sufficiently small values of $\eps$.

Summarizing, we have

\begin{theorem}\label{ThM12bis}
Assume $m\in(1,2)$ and $\kappa\in C^1(\gamma)$ such that \eqref{gg} holds.
Let $\lambda$ be an eigenvalue of  $\cD$ with multiplicity $J$, i.e., $\lambda=\lambda_j=\lambda_{j+1}=\cdots=\lambda_{j+J-1}$ and $\lambda_{j-1}<\lambda<\lambda_{j+J}$.
Then the total multiplicity of eigenvalues of $\cA_\eps$ that lie in the interval
\begin{equation*}
   \Delta_\eps=\left[\lambda-2J C\eps^{ \beta(m,\eta)},\lambda+2J C\eps^{ \beta(m,\eta)}\right]
\end{equation*}
is equal to $J$. Here $ \beta(m,\eta)$ is given by \eqref{BetaNM}, where $\eta$ is any positive number such that
$m-1<\eta<1.$

In addition, if $\lambda$ is a simple eigenvalue of $\cD$ with an eigenfunction $u$, $\|u\|=1$, and the interval $\Delta_\eps=\left[\lambda-2J C\eps^{ \beta(m,\eta)+ \tau},\lambda+2J C\eps^{ \beta(m,\eta)+\tau}\right]$ for a certain $\tau>0$ contains only an eigenvalue  of  $\cA_\eps$,  then there exists a sequence of eigenfunction $u_\eps$ of $\cA_\eps$ such that $u_\eps\to u$ in $\cH$ weakly.
\end{theorem}

\noindent
{\bf Acknowledgements} \quad
This work has been partially supported by Grant PID2022-137694NB-I00 funded by MICIU/AEI/10.13039/501100011033 and by ERDF/EU.

\vspace{0.2cm}

\noindent
{\bf Conflict of interest} The authors have not disclosed any conflict of interest.

\end{document}